\documentclass[11pt]{amsart}
\usepackage[utf8]{inputenc}
\usepackage{xcolor}
\usepackage{tikz}
\usepackage{bm}
\usepackage{amsthm, amsfonts, amsmath, amssymb, mathrsfs}

\title[Rigidity of hypersurface Zoll manifolds]{Topological and spectral rigidity of hypersurface Zoll manifolds}
\author{Gustavo Martins}

\address{IMPA - Instituto Nacional de Matematica Pura e Aplicada, Rio de Janeiro, RJ, Brasil, 22460-320.}
\email{\texttt{gustavo.martins@impa.br}}

\newtheorem{theorem}{Theorem}[section]
\newtheorem{proposition}[theorem]{Proposition}
\newtheorem{corollary}[theorem]{Corollary}
\newtheorem{lemma}[theorem]{Lemma}
\theoremstyle{definition}
\newtheorem{definition}[theorem]{Definition}
\newtheorem{example}[theorem]{Example}
\theoremstyle{remark}
\newtheorem{remark}[theorem]{Remark}
\newtheoremstyle{named}{}{}{\itshape}{}{\bfseries}{.}{.5em}{\thmnote{#3}}
\theoremstyle{named}
\newtheorem*{namedtheorem}{Teorema}

\newcommand{\N}{\mathbb{N}}
\newcommand{\Z}{\mathbb{Z}}

\newcommand{\R}{\mathbb{R}}
\newcommand{\C}{\mathbb{C}}
\newcommand{\I}{\textbf{I}}
\newcommand{\Mass}{\textbf{M}}
\newcommand{\Flat}{\mathcal{F}}
\newcommand{\FF}{\mathbf{F}}
\newcommand{\ZZ}{\mathcal{Z}}

\newcommand{\Proj}{\mathbb{P}}
\newcommand{\Sph}{\mathbb{S}}
\newcommand{\Pcal}{\mathcal{P}}
\newcommand{\Scal}{\mathcal{S}}
\newcommand{\Tcal}{\mathcal{T}}
\newcommand{\Zoll}{\mathscr{Z}}

\begin{document}

\begin{abstract}
    We show that manifolds admitting Zoll families of minimal hypersurfaces are diffeomorphic to $\Sph^{n}$ or to $\R\Proj^{n}$, and we characterize their Zoll families in both cases.
    Then, we prove that a metric on $\R\Proj^3$ is surface Zoll if and only if its projective widths are all equal.
    Our last result asserts that a closed surface with the first four widths equal to $2\pi$ is isometric to $\R\Proj^2$ with its constant curvature one metric.
\end{abstract}

\maketitle

\section{Introduction}

A Zoll manifold is a Riemannian manifold whose geodesics are all simple closed curves and have the same length.
The name comes from Otto Zoll, who discovered many nontrivial examples of metrics with this property on the two-dimensional sphere \cite{Zoll}.

Although rigid, the Zoll condition is flexible enough to allow metrics with a trivial isometry group, as shown by Guillemin \cite{Gui}.
At the same time, both constructions in \cite{Zoll} and \cite{Gui} relate Zoll metrics to odd functions, and none of them induces nontrivial metrics on the real projective plane.
This observation reveals its full scope in Green's Theorem \cite{Gre}, which states that every Zoll metric on $\R\Proj^2$ is isometric to the canonical one up to rescaling (see also \cite{Bes} and \cite{LM}).

In higher dimensions, Weinstein was able to expand the ideas in \cite{Zoll} to obtain many nontrivial Zoll metrics on the $n$-sphere $\Sph^n$ for all $n\geq 3$.
On the other hand, Berger proved that the only Zoll metrics on the real projective $n$-space $\R\Proj^n$ are, up to isometry and rescaling, the canonical.
We refer to \cite{Bes} for all these results.

From the topological perspective, it is easy to prove that any Zoll manifold is closed and has a fundamental group of order at most two.
As a consequence, a Zoll surface is a 2-sphere or a real projective plane.
In dimension three, such manifolds are diffeomorphic to $\Sph^3$ or to $\R\Proj^3$ by Geometrization.

This low-dimensional pattern breaks down in dimensions four and beyond, because the canonical metrics on the complex and quaternionic projective spaces $\C\Proj^n$ and $\mathbb{H}\Proj^n$ and on the Cayley plane $\mathbb{O}\Proj^2$ are all Zoll.
There is, however, no known example with a topology different from that of a compact rank one symmetric space (CROSS).
In fact, it is conjectured that every Zoll manifold is diffeomorphic to one of them.
The best known result is a theorem by Bott and Samelson \cite{Bott,Sam} (see also \cite[Theorem 7.23]{Bes}), which establishes that a space that admits a Zoll metric has the integral cohomology ring of a CROSS.

Another possible generalization of Zoll surfaces to higher dimensions replaces geodesics by minimal hypersurfaces.
Roughly speaking, a hypersurface Zoll family in a Riemannian manifold $(M,g)$ is a smoothly parametrized collection $\Zoll$ of closed embedded minimal hypersurfaces such that through each point $x\in M$ and each tangent hyperplane $\Pi\subset T_xM$ passes a unique $\Sigma\in\Zoll$ determined by $T_x\Sigma = \Pi$ (see Definitions \ref{def: transitive families} and \ref{def: hypersurface Zoll manifolds}).
We call the triple $(M,g,\Zoll)$ a hypersurface Zoll manifold, or simply a surface Zoll manifold when the dimension of $M$ is three.

Recently, Ambrozio, Marques, and Neves \cite{AMN1}, and Ambrozio and Guajardo \cite{AG} proved the existence of many nontrivial hypersurface Zoll metrics on $\Sph^n$ and on $\R\Proj^n$.
In \cite{AMN2}, Ambrozio, Marques, and Neves studied their rigidity, in dimension three, from a geometric point of view.

In this paper, we classify hypersurface Zoll manifolds topologically, and prove a stronger analogous version of the Bott-Samelson Theorem for this setting.
Our work answers in all dimensions a conjecture posed by Ambrozio, Marques and Neves that states that every surface Zoll manifold is diffeomorphic to $\Sph^3$ or to $\R\Proj^3$.

The relevant differential-topological object is now a smooth manifold $M$ equipped with a family $\Zoll$ of closed embedded hypersurfaces such that through each point $x\in M$ and each tangent hyperplane $\Pi\subset T_xM$ passes a unique $\Sigma\in\Zoll$ that depends smoothly on $(x,\Pi)$ and is determined by $T_x\Sigma = \Pi$.
The collection $\Zoll$ is called an unoriented transitive family in this case, and has the structure of a smooth manifold of the same dimension as $M$.
Similarly, we call $\Zoll$ an oriented transitive family if we require that the manifold $M$, the elements of $\Zoll$ and $\Pi$ all be oriented (see the precise Definition \ref{def: transitive families}).

\begin{namedtheorem}[Theorem A]
    A smooth $(n+1)$-dimensional manifold that admits an unoriented transitive family $\Zoll$ of closed embedded hypersurfaces is diffeomorphic to $\Sph^{n+1}$ or to $\R\Proj^{n+1}$.
    In the first case, the elements of $\Zoll$ are diffeomorphic to the $n$-sphere, whereas in the second, to the real projective $n$-space.
\end{namedtheorem}

We note that no compactness assumption on $M$ is required for the above result.
This parallels the fact that Zoll manifolds are closed.

An unoriented transitive family $\Zoll$ in $M$ forces a duality between points and hypersurfaces.
This is stronger when $M$ is the real projective $n$-space, in the sense that it allows the complete characterization of $\Zoll$ by our methods.
When $M$ is the $n$-sphere, the duality that arises depends on a choice of orientation, and our best description of $\Zoll$ is up to homeomorphism.

\begin{namedtheorem}[Theorem B]
    Suppose $\Zoll$ is an unoriented transitive family of closed embedded hypersurfaces in an $(n+1)$-dimensional manifold $M$.
    If $M\approx \R\Proj^{n+1}$, then $\Zoll$ is diffeomorphic to $\R\Proj^{n+1}$ and has itself an unoriented transitive family of embedded real projective $n$-spaces parametrized by $M$.
    If $M\approx \Sph^{n+1}$, then $\Zoll$ is homeomorphic to $\R\Proj^{n+1}$ and its universal cover has an oriented transitive family of embedded $n$-spheres parametrized by $M$.
\end{namedtheorem}

We suspect that, even when $M\approx \Sph^{n+1}$, $\Zoll$ is diffeomorphic to $\R\Proj^{n+1}$.

It is important to remark that we define oriented and unoriented transitive families of hypersurfaces on manifolds of dimension $n+1\geq 2$, but we only prove Theorems A and B for $n+1\geq 3$.
Both Theorems A and B hold when $n+1=2$, and are proved in the work of LeBrun and Mason \cite{LM} (see Remark \ref{rmk: the case n+1=2} below).

A hypersurface Zoll family in a Riemannian manifold is an unoriented transitive family of closed embedded minimal hypersurfaces.
Therefore, Theorems A and B immediately classify their topology up to diffeomorphism.

\begin{namedtheorem}[Corollary]
    A hypersurface Zoll $(n+1)$-manifold $(M,g,\Zoll)$ is diffeomorphic to the $(n+1)$-sphere or to the real projective $(n+1)$-space.
    In the first case, the Zoll hypersurfaces are $n$-spheres, $\Zoll$ is homeomorphic to $\R\Proj^{n+1}$ and its universal cover has an oriented transitive  family of embedded $n$-spheres parametrized by $M$.
    In the second, the Zoll hypersurfaces are real projective $n$-spaces, the family $\Zoll$ is diffeomorphic to $\R\Proj^{n+1}$ and has itself a transitive family of embedded projective $n$-spaces parametrized by $M$.
\end{namedtheorem}

In dimension three, a surface Zoll family will always be diffeomorphic to the real projective $3$-space.
In fact, our results and the G\'alvez-Mira uniqueness theorem \cite{GM} completely describe the topology of surface Zoll manifolds.

\begin{namedtheorem}[Corollary]
    Let $g$ be a surface Zoll metric on a smooth 3-manifold $M$.
    Then there is a unique surface Zoll family $\Zoll$ relative to $g$.
    When $M$ is $\Sph^3$, all its immersed minimal 2-spheres are embedded and belong to $\Zoll$.
    When $M$ is $\R\Proj^3$, every immersed minimal 2-sphere is a double cover of an element of $\Zoll$.
\end{namedtheorem}

Therefore, we can refer to a surface Zoll manifold as a pair $(M,g)$ instead of a triple $(M,g,\Zoll)$.
It would be interesting to prove the uniqueness of hypersurface Zoll families in all dimensions.
One difficulty is that even with the canonical metric on $\Sph^4$, there exist minimal 3-spheres that are not equators.

After the classification of hypersurface Zoll manifolds up to diffeomorphism, the second part of this work continues the geometric characterization started in \cite{AMN2}.
There, Ambrozio, Marques, and Neves defined four min-max numbers, the spherical widths (see also \cite{HasKet}), associated to a Riemannian metric on $\Sph^3$, and proved that a metric is surface Zoll if and only if the equality of these four numbers holds.

It is also possible to define four projective widths associated with a metric $g$ on the real projective $3$-space.
They are denoted as $0<\sigma_0(\R\Proj^3,g)\leq\cdots\leq\sigma_3(\R\Proj^3,g)$, and $\sigma_0(\R\Proj^3,g)$ is the least area of an embedded $\R\Proj^2$.
In analogy with the Ambrozio-Marques-Neves Theorem, we prove a rigidity result for the real projective $3$-space.

\begin{namedtheorem}[Theorem C]
    A Riemannian metric $g$ on $\R\Proj^3$ is surface Zoll if and only if
    \[ \sigma_0(\R\Proj^3,g) = \sigma_1(\R\Proj^3,g) = \sigma_2(\R\Proj^3,g) = \sigma_3(\R\Proj^3,g). \]
\end{namedtheorem}

Going beyond widths associated to the area of surfaces with fixed topology in three-manifolds, hypersurface Zoll manifolds are likely to appear as rigid cases in the Almgren-Pitts min-max theory of minimal hypersurfaces.
In this theory, one works in a closed Riemannian manifold $(M,g)$ and obtains the volume spectrum $\{\omega_p(M,g)\}_{p\in\N}$.
When the ambient dimension is between three and seven, each $p$-width is realized as a sum
\[ \omega_p(M,g) = \sum_{i=1}^{k_p} m^{(p)}_i area\big(\Sigma^{(p)}_i\big) \]
for some disjoint collection of closed smoothly embedded minimal hypersurfaces $ \Sigma^{(p)}_1,...,$ $\Sigma^{(p)}_{k_p}$ with respective integer multiplicities $m^{(p)}_1,...,$ $m^{(p)}_{k_p}$.
In higher dimensions, the same description is true, but each $\Sigma^{(p)}_i$ might have singularities of Hausdorff codimension at least seven (see \cite{Pitts}, and also \cite{MN-Ricci}).

In analogy to the Laplace spectrum, the volume spectrum of a manifold also satisfies a Weyl law \cite{LMN}, which is fundamental in the solution of the Yau Conjecture (see \cite{IMN}, \cite{MN-Ricci}, \cite{Song}, and \cite{Zhou}).
For generic metrics, the volume spectrum is a strictly increasing sequence of positive real numbers, and its asymptotic behavior is one of the few informations one has.

When the metric has two equal consecutive widths, however, the construction of many minimal hypersurfaces becomes simpler and relies on a Lusternik-Schnirelman argument (see \cite{MN-Ricci}).
This idea can be refined when more widths coincide, and the philosophy is that the equality $\omega_1=\omega_{n+1}$ for $n=\dim M$ should imply the existence of an $n$-sweepout of $(M,g)$ by closed minimal hypersurfaces, a property that resembles the hypersurface Zoll condition.

This idea was made precise in \cite{AMN2} in dimension two.
There, they proved that a closed Riemannian surface $(M,g)$ with $\omega_1(M,g)  = \omega_3(M,g)$ is either a Zoll sphere or a real projective plane with systole equal to $\omega_1(M,g)/2$.
It then follows as a consequence of Pu's systolic inequality \cite{Pu}, the Weyl law, and a theorem of Weinstein \cite{Wein} that the real projective plane with its canonical metric is determined by its volume spectrum up to isometry.

As remarkable as it is, the relation between the widths and the existence of Zoll metrics suggests that it should be possible to prove the spectral rigidity of the real projective plane using Green's Theorem instead of Pu's inequality.
This new route is supported by the computations done in \cite{Marx-Kuo}.
There, Marx-Kuo computed all the widths of $\R\Proj^2$ with its canonical metric, and showed, as a consequence, that the first five of them are equal.

\begin{namedtheorem}[Theorem D]
    A closed Riemannian surface $(M,g)$ is isometric to the real projective plane with its constant curvature one metric if and only if 
    \[ \omega_1(M,g) = \omega_2(M,g) = \omega_3(M,g) = \omega_4(M,g) = 2\pi. \]
\end{namedtheorem}

We note that, at the time of writing this paper, $\Sph^n$ is the only known closed manifold that admits a metric with equal first $n+1$ widths for $n\geq 3$.
Optimistically, one may suspect that it is the only one (see also \cite{BA} for what is known about $\R\Proj^3$).
Whether this is true or not, we hope that our results will help future work on this problem. 

\subsection*{Organization of the paper}

In Section 2, we prove Theorem A.
There, we give the precise definitions of (un)oriented transitive families, show some examples and classify their topologies up to diffeomorphism.
We separate the proof in two cases, depending on whether the elements of the unoriented transitive family are two-sided or one-sided.
Theorem A follows by the combination of Theorem \ref{thm: classification of transitive families of two-sided hypersurfaces} and Theorem \ref{thm: classification of transitive families of one-sided hypersurfaces}.

Section 3 is devoted to the proof of Theorem B.
The case $M\approx \R\Proj^{n+1}$ turns out to be the simplest, and we give the classification up to diffeomorphism.
When $M\approx \Sph^{n+1}$, the arguments used in Section 2 do not apply, so we can only give the classification up to homeomorphism.
Theorem B is a direct consequence of Corollary \ref{cor: Theorem B for M=RP^(n+1)}, Proposition \ref{prop: psi is a covering map} and Proposition \ref{prop: P is homeomorphic to RP^(n+1) when M=S^(n+1)}.

Section 4 starts the geometric part of the paper.
There, we recall some notation from min-max theory, give the definition of a hypersurface Zoll manifold and prove that every surface Zoll metric on $\R\Proj^3$ has equal projective widths.
We also discuss the Corollaries following Theorems A and B.
Section 5 continues with the proof of the last implication remaining to conclude Theorem C.

Section 6 finishes the paper with the proof of Theorem D.
That the canonical metric has equal first four widths was shown in \cite{Marx-Kuo}.
We prove the other implication.

\subsection*{AI Disclosure} We used the free online version of Anthropic's Claude to generate the TikZ code for Figure 1 and we adapted it to obtain Figure 2.
Google's Gemini was used to help with the structural formatting of the references list and with some minor spell-checking and grammatical corrections.

\subsection*{Acknowledgments}
I am very grateful to my advisor Lucas Ambrozio, who supported, taught, and inspired me over the past few years, and for all the valuable meetings we had.
I would like to thank Andr\'e Neves for several discussions, and for hosting me at the University of Chicago during the third year of my Ph.D.
I also thank Diego Guajardo for all the mathematical conversations we had.
While doing this research, I was supported by CAPES, under the scholarship of the Programa Doutorado-sandu\'iche no Exterior (88881.126120/2025-01), and by CNPq, under the doctoral scholarship 141382/2023-6.


\section{Manifolds with transitive families of hypersurfaces}

In this section, all manifolds and submanifolds are assumed to be smooth, connected and without boundary unless otherwise stated.

Let $M$ be a manifold of dimension $n+1\geq 2$.
Denote by $G_n(M)$ and $G^+_n(M)$ the Grassmanian bundles of unoriented and of oriented hyperplanes, respectively.
Write $\mu:G_n(M)\to M$ and $\mu^+:G^+_n(M)\to M$ for the canonical projections.

For any immersed hypersurface $\Sigma$ in $M$, we define its \textit{(unoriented) Legendrian lift} as $\mathcal{L}_\Sigma:=\{(x,T_x\Sigma):x\in\Sigma\}\subset G_n(M)$.
When $\Sigma$ is oriented, we define its \textit{oriented Legendrian lift} as $\mathcal{L}^+_\Sigma:=\{(x,T_x\Sigma):x\in\Sigma\}\subset G^+_n(M)$, where we understand $T_x\Sigma$ with the induced orientation of $\Sigma$.

The next definition resembles the one in \cite{GM}, except for item (iii) below, which is closer to the description in \cite{LM}.

\begin{definition}
    An \textit{unoriented transitive family of hypersurfaces} in $M$ is a collection $\Zoll$ of immersed hypersurfaces with the following three properties:
    \begin{itemize}
        \item[(i)] The Legendrian lift of every $\Sigma\in\Zoll$ is embedded in $G_n(M)$.
        \item[(ii)] The set $\{\mathcal{L}_\Sigma:\Sigma\in\Zoll\}$ is a smooth foliation of $G_n(M)$.
        In particular, for every $(x,\Pi)\in G_n(M)$, there is a unique $\Sigma = \Sigma(x,\Pi)\in \Zoll$ such that $T_x\Sigma = \Pi$.
        \item[(iii)] The foliation $\{\mathcal{L}_\Sigma:\Sigma\in \Zoll\}$ is \textit{locally trivial}, in the sense that for any fixed $\Sigma\in \Zoll$, there exists a neighborhood $U$ of $\mathcal{L}_\Sigma$ in $G_n(M)$ that is diffeomorphic to the product $\Sigma\times \R^{n+1}$ in such a way that $\mathcal{L}_\Sigma$ corresponds to $\Sigma\times\{0\}$ and every $\Sigma\times\{\textnormal{pt}\}$ is the image of the Legendrian lift of some $\Sigma'\in\Zoll$.
    \end{itemize}
    When $M$ is oriented, we say that a collection of immersed oriented hypersurfaces is an \textit{oriented transitive family of hypersurfaces} in $M$ if it has the three properties above, but with $G_n(M)$ and $\mathcal{L}_\Sigma$ replaced respectively by $G^+_n(M)$ and $\mathcal{L}^+_\Sigma$.
    \label{def: transitive families}
\end{definition}

For our applications, we will assume that all elements of a transitive family $\Zoll$ are closed.
We will always make this assumption explicit.

Let $\Zoll$ be an unoriented (resp. oriented) transitive family of hypersurfaces in $M$.
As an immediate consequence, the set $P$ (resp. $P^+$) of leaves of the foliation $\{\mathcal{L}_\Sigma:\Sigma\in \Zoll\}$ (resp. $\{\mathcal{L}^+_\Sigma:\Sigma\in \Zoll\}$) has the structure of a connected smooth manifold of dimension $n+1$.
Moreover, the map $\nu:G_n(M)\to P$ (resp. $\nu^+:G^+_n(M)\to P^+$) that collapses the leaves induces a smooth fiber bundle with fibers diffeomorphic to elements of $\Zoll$.

We call $P$ (resp. $P^+$) the \textit{parameter space} of the transitive family $\Zoll$.

The pair of projections $\mu:G_n(M)\to M$ and $\nu:G_n(M)\to P$ (resp. $\mu^+:G^+_n(M)\to M$ and $\nu^+:G^+_n(M)\to P^+$) forms a double fibration.
This means that both are fiber bundles and their derivatives are orthogonal.
In mathematical notation, this writes as $\ker\mu_*\cap\ker\nu_* = 0$ (resp. $\ker\mu^+_*\cap\ker\nu^+_* = 0$).

When $\Zoll$ is unoriented, there exists a duality between $M$ and $P$ close to the projective duality between points and hyperplanes in projective geometry.
Any point $\sigma\in P$ corresponds to the hypersurface $\Sigma_\sigma := \mu(\nu^{-1}(\sigma))\in\Zoll$, and any point $x\in M$, to the immersed real projective $n$-space $\Sigma^*_x:= \nu(\mu^{-1}(x))$ in $P$.

When $\Zoll$ is oriented, the duality that arises between $M$ and $P^+$ resembles that of points and oriented equators on the $(n+1)$-sphere.
Any point $\sigma\in P^+$ corresponds to an oriented immersed hypersurface $\Sigma_\sigma\subset M$; and a point $x\in M$ corresponds to an oriented immersed sphere $\Sigma^+_x:=\nu^+\left((\mu^+)^{-1}(x)\right)$.
We note that, since $M$ is oriented, $G^+_n(M)$ is orientable, and henceforth we fix an orientation of $G^+_n(M)$, which is equivalent to taking smoothly compatible orientations on the fibers of $G^+_n(M)$.
The orientation of $\Sigma^+_x$ comes from this.

By construction, the map $\sigma\in P\mapsto \Sigma_\sigma\in\Zoll$ (resp. $\sigma\in P^+\mapsto \Sigma_\sigma\in\Zoll$) is a bijection, and we make $\Zoll$ into an $(n+1)$-manifold by forcing it to be a diffeomorphism.
We henceforth write $\Zoll = \{\Sigma_\sigma:\sigma\in P\}$ (resp. $\Zoll = \{\Sigma_\sigma:\sigma\in P^+\}$).

The apparent redundancy in the notation comes naturally when dealing with examples.
We present some that are relevant for our discussion, and refer the reader to \cite{GM} for other constructions.

\begin{example}
    The set of equators in $\Sph^{n+1}$.
    For every vector $v\in \Sph^{n+1}$, we define the equator $S_v:=\{x\in \Sph^{n+1}:\langle x,v \rangle = 0\}$.
    Since $S_v = S_w$ if and only if $v = \pm w$, we write $S_{[v]}$ for $[v]\in\R\Proj^{n+1}$ the line in $\R^{n+2}$ generated by $v$.
    Thus, the family of (unoriented) equators, $\{S_{\sigma}:\sigma\in\R\Proj^n\}$, is an unoriented transitive family of $n$-spheres.
\end{example}

\begin{example}
    The set of hyperplanes in $\R\Proj^{n+1}$.
    Each of its elements is the image $\Proj_\sigma = \{[x]:x\in S_\sigma\}$ of $S_\sigma$ under the projection $x\in \Sph^{n+1}\mapsto[x]\in\R\Proj^{n+1}$.
    Hence, this transitive family is $\{\Proj_\sigma:\sigma\in\R\Proj^{n+1}\}$, also parameterized by $\R\Proj^{n+1}$.
\end{example}

\begin{example}
    Think of $v\in \Sph^{n+1}$ as a unit normal vector field along the equator $S_v$.
    If we fix an orientation on $\Sph^{n+1}$, $S_v$ has an induced orientation by $v$, which is the opposite orientation of $S_{-v}$.
    We call $S_v$ an oriented equator and the set $\{S_v:v\in \Sph^{n+1}\}$ is an oriented transitive family on $\Sph^{n+1}$.
\end{example}

\begin{example}
    Let $(M,g)$ be an oriented complete $(n+1)$-dimensional Riemannian manifold of bounded geometry, and fix some number $r>0$ smaller than its injectivity radius.
    For each $p\in M$, denote by $S_r(p) := \{x\in M :d_g(x,p) = r\}$ the geodesic sphere of radius $r$ centered at $p$.
    By our restriction on $r$, the collection $\Zoll_r := \{S_r(p):p\in M\}$ is an oriented transitive family of closed embedded $n$-spheres in $M$ parametrized by $M$.
    Since every oriented manifold has a complete metric of bounded geometry (see \cite{MN15}), every such manifold has an oriented transitive family of closed embedded hypersurfaces.
    \label{example: every oriented manifold has an oriented transitive family of closed hypersurfaces}
\end{example}

\subsection{Transitive families of two-sided hypersurfaces}
The main goal of this section is to classify manifolds that admit unoriented transitive families of closed embedded hypersurfaces.
Our discussion only deals with the case $n+1\geq3$, and we make this assumption on $n$ throughout the rest of this section.
From now on, a transitive family will always be assumed unoriented unless otherwise stated.

We start by classifying the manifolds with transitive families of closed two-sided hypersurfaces.

Let $\Zoll = \{\Sigma_\sigma\}_{\sigma\in P}$ be any fixed transitive family of closed two-sided embedded hypersurfaces in a manifold $M$.
Denote the induced double fibration by $\mu:G_n(M)\to M$ and $\nu:G_n(M)\to P$.

Since each $\Sigma\in\Zoll$ is two-sided, its normal bundle $N\Sigma:= TM|_{\Sigma}/T\Sigma$ is trivial.
And since $\mu:G_n(M)\to M$ embeds the fibers $\nu^{-1}(\sigma)$ onto $\Sigma_\sigma$, the derivative $\mu_*:TG_n(M)\to \mu^*TM$ induces an isomorphism from the quotient bundle $L = TG_n(M)/(\ker(\mu_*)\oplus \ker(\nu_*))$ to the bundle $N\Zoll$ of all triples $(x,\Pi,\xi)$ for which $\xi\in N_x\Sigma_\sigma$, where $\sigma = \nu(x,\Pi)$.

This identifies a tangent vector $\zeta\in T_\sigma P$ with a section of $N\Sigma_\sigma$ as follows.
For any fixed $\sigma\in P$ and each $x\in \Sigma_\sigma$, the derivative of $\nu:G_n(M)\to P$ induces an isomorphism $T_{(x,T_x\Sigma_\sigma)}G_n(M)/\ker(\nu_*)\simeq T_\sigma P$.
Taking a further quotient $T_{(x,T_x\Sigma_\sigma)}G_n(M)/\ker(\nu_*)\to T_{(x,T_x\Sigma_\sigma)}G_n(M)/(\ker(\mu_*)\oplus\ker(\nu_*))\simeq N_x\Sigma_\sigma$, we obtain a map $T_\sigma P\to N_x\Sigma_\sigma$ that sends $\zeta$ to some $\overline{\zeta}_x$.
We then vary $x$ in $\Sigma_\sigma$ to obtain a section $\overline{\zeta}\in \Gamma(N\Sigma_\sigma)$ given by $\overline{\zeta}(x) = \overline{\zeta}_x$, and this establishes a well-defined linear map $T_\sigma P\to \Gamma(N\Sigma_\sigma)$.
We call $\overline{\zeta}$ the \textit{infinitesimal variations of $\Sigma_\sigma$ along $\Zoll$ in the direction of $\zeta$}.

For each $\zeta\in T_\sigma P$, there is a set $S_\zeta := \{x\in \Sigma_\sigma:\overline{\zeta}_x = 0\}$.
In fact, since $\overline{\zeta}_x = 0$ if and only if $\lambda\overline{\zeta}_x = 0$ for every $\lambda\neq 0$, we see that $S_\zeta = S_{\lambda\zeta} =: S_{[\zeta]}$, where $[\zeta]\in \Proj T_\sigma P$.
We call the sets $S_\xi$, $\xi\in\Proj T_\sigma P$, the \textit{nodal sets} or the \textit{nodal hypersurfaces} of $\Sigma_\sigma$, and write $S^\sigma_\xi = S_\xi$ to emphasize the dependence of $S_\xi$ on $\sigma\in P$ whenever necessary.

We claim that each $S_\xi$ is a submanifold of $\Sigma_\sigma$ of codimension one.
To show this, fix a point $x\in \Sigma_\sigma$ and a Riemannian metric $g$, and define the homomorphism
\begin{equation}
     E^g_{x,\sigma}:\zeta\in T_\sigma P \mapsto (\overline{\zeta}(x), \nabla\overline{\zeta}(x))\in N_x\Sigma_\sigma\times \left(T^*_x\Sigma_\sigma\otimes N_x\Sigma_\sigma\right),
     \label{map: the isomorphism E_x,sigma}
\end{equation}
where $\nabla\overline{\zeta} = \nabla^g\overline{\zeta}$ is the covariant derivative of $\overline{\zeta}$ with respect to our choice of metric $g$.
\begin{proposition}
    $E^g_{x,\sigma}$ is an isomorphism for any metric $g$ in $M$.
    \label{prop: Zoll dimension reduction}
\end{proposition}
\begin{proof}
    We only need to prove that $E^g_{x,\sigma}$ is an isomorphism for some $g$ to conclude it is an isomorphism for every metric on $M$.
    
    Take $g$ so that $\Sigma_\sigma$ is a totally geodesic hypersurface.
    We may even take $g$ to be the product metric $g = g_\sigma + dt^2$ in some tubular neighborhood $U\approx \Sigma_\sigma\times (-1,1)$ of $\Sigma_\sigma$, where $g_\sigma$ is some metric in $\Sigma_\sigma$.
    We then identify $N\Sigma_\sigma = T^\perp\Sigma_\sigma$ and choose a unit normal $N$ along $\Sigma_\sigma$.
    In particular, the associated variational field $\overline{\zeta}$ of a vector $\zeta\in T_\sigma P$ is decomposed as $\overline{\zeta} = \phi_\zeta\cdot N$, where $\phi_\zeta \in C^\infty(\Sigma_\sigma)$.

    Note that, to prove the proposition for $g$, it suffices to show that the homomorphism
    \begin{equation}
         E^g_{x,\sigma}:\zeta\in T_\sigma P \mapsto \left( \phi_\zeta(x),d\phi_\zeta(x) \right) \in \R\times T^*_x\Sigma_\sigma.
         \label{map: the isomorphism E^g_x,sigma}
    \end{equation}
    is an isomorphism.

    Observe that both domain and co-domain of $E^g_{x,\sigma}$ have the same dimension, and that the composition of $E^g_{x,\sigma}$ with the projection $\R\times T_x\Sigma_\sigma\to \R$ is onto, because the homomorphism $\zeta\in T_\sigma P\mapsto \overline{\zeta}_x\in T^\perp_x\Sigma_\sigma$ is surjective.
    Therefore, we only need to show that, for every vector $v\in T_x\Sigma_\sigma$, there is some $\zeta\in T_\sigma P$ whose corresponding function $\phi_\zeta$ vanishes at $x$ and has gradient $\nabla\phi_\zeta(x) = v$.
    Here, we write $\nabla\phi_\zeta = \nabla^{\Sigma_\sigma}\phi_\zeta$ for the gradient in the induced metric by $g$.
    
    For now, fix a non-zero vector $v\in T_x\Sigma_\sigma$, and write $V = \{w\in T_x\Sigma_\sigma:g(v,w)=0\}$ for the orthogonal complement.
    For all $t\in\R$ with $|t|$ sufficiently small, the hypersurface $\Sigma_t:= \Sigma_{\sigma(t)} \in\Zoll$ that passes through $x$ with $T_x\Sigma_t = V\oplus\langle v\cos t + N_x\sin t \rangle$ lies in the tubular neighborhood $U_\epsilon = \{\exp_x(sN_x):x\in \Sigma_\sigma, |s|<\epsilon\} \approx \Sigma_\sigma\times (-\epsilon,\epsilon)$ and can be written as the graph of a function $u_t\in C^\infty(\Sigma_\sigma)$ that depends smoothly on $t$.
    Explicitly, we write $\Sigma_t = \{\exp_x(u_t(x)N_x):x\in \Sigma_\sigma\}$.
    This implies that the hyperplane $T_x\Sigma_t\subset T_xM$ is the set of vectors of the form $w+(\nabla_wu_t)N_x$ for $w\in T_x\Sigma_\sigma$.
    By the equality $T_x\Sigma_t = V\oplus\langle v\cos t + N_x\sin t \rangle$, we conclude that $\nabla_wu_t = 0$ for every $w\in V$ and $v+(\nabla_v u_t)N_x = \lambda (v\cos t + N_x\sin t)$ for some $\lambda\neq 0$.
    The only possible $\lambda$ is $(\cos t)^{-1}$, which implies $\nabla_v u_t = \tan t$.
    
    By construction, $\sigma = \sigma(0)$, and the choice of $\zeta = \sigma'(0)\in T_\sigma P$ has an associated function equal to
    \[ \phi_\zeta = \frac{d}{dt}\bigg|_{t=0} u_t. \]
    By smoothness, we see that
    \[ \nabla_w\phi_\zeta = \frac{d}{dt}\bigg|_{t=0} \nabla_w u_t = 0 \]
    for all $w\in V$, and
    \[ \nabla_v\phi_\zeta = \frac{d}{dt}\bigg|_{t=0} \nabla_v u_t = \frac{d}{dt}\bigg|_{t=0}\tan t = 1. \]
    This proves that $\nabla\phi_\zeta(x) = v$ as desired.
\end{proof}

As an abuse of notation, we write $E_{x,\sigma}$ for the maps described in \eqref{map: the isomorphism E_x,sigma} and \eqref{map: the isomorphism E^g_x,sigma}, and also for the map $\zeta\mapsto(\phi_\zeta(x),\nabla\phi_\zeta(x))$, where $\nabla \phi_\zeta$ is the $g$-gradient of $\phi_\zeta$ and $\phi_\zeta = \phi^g_\zeta$ is defined as in the proof of Proposition \ref{prop: Zoll dimension reduction}

As a consequence of Proposition \ref{prop: Zoll dimension reduction}, we have proved that, for every $\sigma\in P$ and $x\in\Sigma_\sigma$, and for every $(n-1)$-dimensional subspace $V\subset T_x\Sigma_\sigma$ there is a unique $[\zeta]\in \Proj T_\sigma P$ such that $T_x S_{[\zeta]} = V$.
Each $S_{[\zeta]}$ is a smooth closed, but possibly disconnected hypersurface of $\Sigma_\sigma$.
Furthermore, since the correspondence $\zeta\mapsto\overline{\zeta}$ is smooth, so is $[\zeta]\mapsto S_{[\zeta]}$.

We can say more about the nodal sets of $\Sigma_\sigma$.
In fact, we are able to determine the topology of both $\Sigma_\sigma$ and its nodal sets up to diffeomorphism.

Once again, fix a Riemannian metric $g$ on $M$ and $\sigma\in P$.
Choose a basis $\zeta_0,...,\zeta_{n}\in T_\sigma P$ and write $\phi_i:= \phi_{\zeta_i}$ for $i=0,...,n$, where we use the notation $\zeta\mapsto\phi_\zeta$ from the proof of Proposition \ref{prop: Zoll dimension reduction}.
We define the map $\widetilde{\psi}_\sigma:\Sigma_\sigma\to \R^{n+1}$ that assigns $x\in \Sigma_\sigma$ to $\widetilde{\psi}_\sigma(x) = (\phi_0(x),...,\phi_n(x))$.

If $\widetilde{\psi}_\sigma(x) = 0$ for some point $x\in \Sigma_\sigma$, then the corresponding linear map $E_{x,\sigma}:T_\sigma P\to \R\times T_x\Sigma_\sigma$ would have a nontrivial kernel, and would not be an isomorphism.
Since this cannot happen, $\widetilde{\psi}_\sigma$ maps $\Sigma_\sigma$ in $\R^{n+1}\setminus\{0\}$.
Therefore, we can define $\psi_\sigma:x\in\Sigma_\sigma\mapsto\widetilde{\psi}_\sigma(x)/|\widetilde{\psi}_\sigma(x)|\in \Sph^n$.

\begin{proposition}
    For any choice of basis $\zeta_0,...,\zeta_n$ of $T_\sigma P$, $\psi_\sigma$ is a diffeomorphism that identifies the family $\{S^\sigma_\xi:\xi\in \Proj T_\sigma P\}$ with the set of equators of $\Sph^n$.
    \label{prop: Sigma is diffeomorphic to S^n through psi_sigma}
\end{proposition}
\begin{proof}
    Observe that the differential of $\widetilde{\psi}_\sigma$ is
    \[ D\widetilde{\psi}_\sigma(x)\cdot v = \left( \langle \nabla\phi_0(x),v \rangle, ..., \langle \nabla\phi_n(x),v\rangle\right) \]
    for any $x\in\Sigma_\sigma$ and $v\in T_x\Sigma_\sigma$.
    If there was a point $x\in\Sigma_\sigma$ and a nonzero vector $v\in T_x\Sigma_\sigma$ for which $D\widetilde{\psi}_\sigma(x)\cdot v = 0$, then $\langle\nabla\phi_i(x),v\rangle = 0$ for all $i=0,...,n$.
    Since $\phi_i = \phi_{\zeta_i}$ and $\zeta_0,...,\zeta_n$ is a basis for $T_\sigma P$, the map $E_{x,\sigma}$ would have image contained in the subspace $\R\times \langle v\rangle^\perp$, where $\langle v\rangle^\perp = \{w\in T_x\Sigma_\sigma: g(v,w) = 0\}$ is the orthogonal complement relative to $g$.
    Hence, $E_{x,\sigma}$ would not be an isomorphism | a contradiction.

    Therefore, the map $\widetilde{\psi}_\sigma$ is an immersion.
    We now claim that $\psi_\sigma$ is as well.
    It suffices to prove that, for every $x\in\Sigma_\sigma$, the image of $D\widetilde{\psi}_\sigma(x)$ and the kernel of the derivative, at $\widetilde{\psi}_\sigma(x)$, of the projection $p\in \R^{n+1}\setminus\{0\}\mapsto p/|p|\in \Sph^n$ have a trivial intersection.

    Indeed, for a fixed $x\in\Sigma_\sigma$ and any choice of orthonormal basis $e_1,...,e_n\in T_x\Sigma_\sigma$, we can take $\zeta_0,...,\zeta_n$ so that $\phi_0(x) = 1$, $\nabla\phi_0(x) = 0$, $\phi_i(x) = 0$ and $\nabla\phi_i(x) = e_i$ for $i=1,...,n$.
    This is because $E_{x,\sigma}$ is an isomorphism.
    In particular, the determinant
    \begin{equation}
        \begin{vmatrix}
            \phi_0(x) & \nabla_{e_1}\phi_0(x) & \cdots & \nabla_{e_n}\phi_0(x) \\
            \vdots & \vdots & \ddots & \vdots \\
            \phi_n(x) & \nabla_{e_1}\phi_n(x) & \cdots & \nabla_{e_n}\phi_n(x) \\
        \end{vmatrix}
        = 1 \neq 0,
        \label{eq: nonzero determinant}
    \end{equation}
    where we denote by $\nabla_vf = df(v) = \langle\nabla f,v \rangle$.
    This is precisely what we wanted, because the first column of the above matrix  is the kernel of the derivative, at $\widetilde{\psi}_\sigma(x)$, of the projection $\R^{n+1}\setminus\{0\}\to \Sph^n$, while the other columns form a basis for the image of $D\widetilde{\psi}_\sigma(x)$.

    What we proved was that $\widetilde{\psi}_\sigma$ is an immersion at a point $x$ for a specific choice of basis of $T_\sigma P$ that depends on $x$.
    A change in bases leads to the taking of a conjugate of the matrix whose determinant appears in \eqref{eq: nonzero determinant}.
    This still implies that the resulting determinant is nonzero, and therefore $\psi_\sigma$ is an immersion at all points $x\in \Sigma_\sigma$ and for any choice of basis of $T_\sigma P$.

    Since $\psi_\sigma:\Sigma_\sigma\to \Sph^n$ is an immersion from a closed $n$-dimensional manifold to the $n$-sphere, $\psi_\sigma$ is a covering map.
    It is a diffeomorphism, because $\Sph^n$ is simply connected.

    To finish the proof, we fix a choice of basis $\zeta_0,...,\zeta_n\in T_\sigma P$ and consider the inner product that makes this into an orthonormal basis.
    For any unit vector $\zeta = u_0\zeta_0+\cdots + u_n\zeta_n\in T_\sigma P$, where $u = (u_0,...,u_n)\in \Sph^n$, the map $\psi_\sigma$ sends $S_{[\zeta]} = \{ \phi_\zeta = 0 \}\subset \Sigma_\sigma$ to the equator $S_{[u]} = \{ v\in \Sph^n:\langle u,v\rangle = 0 \}\subset \Sph^n$.
    This happens because the linearity of $\eta\mapsto\phi_\eta$ allows us to write $\phi_\zeta = u_0\phi_0 + \cdots + u_n\phi_n$, and the points in the zero locus of $\phi_\zeta$ are those where $\langle u,\psi_\sigma\rangle = \langle u,\widetilde{\psi}_\sigma\rangle = u_0\phi_0+\cdots\cdots u_n\phi_n = 0$.
\end{proof}

\begin{corollary}
    The collection $\{S^\sigma_\xi:\xi\in\Proj T_\sigma P\}$ is an unoriented transitive family of $(n-1)$-spheres in $\Sigma_\sigma$ for each $\sigma\in P$.
\end{corollary}

\begin{remark}
    We used the lower bound $n+1 \geq 3$ to ensure that $\Sph^n$ is simply connected and thus concluded that $\psi_\sigma$ is a diffeomorphism.
    If we consider the case $n+1=2$, Proposition \ref{prop: Zoll dimension reduction} still holds and we can follow the arguments of Proposition \ref{prop: Sigma is diffeomorphic to S^n through psi_sigma} up to the conclusion that $\psi_\sigma$ is a covering map.
    It is possible to prove Proposition \ref{prop: Sigma is diffeomorphic to S^n through psi_sigma} for $n+1=2$, and this follows from the arguments in \cite[Section 2]{LM}).

    To properly use the framework in \cite{LM}, we need to show that an unoriented transitive family of simple closed curves $\Zoll$ in a surface $M^2$ is the set of geodesics of some Zoll projective structure as in \cite[Definition 2.4]{LM}.
    We do this in the following way.
    Choose a metric $g$ on $M$, and construct a connection $\nabla$ (not necessarily equal to the Riemannian connection $\nabla^g$) so that, if $\ell\in \Zoll$, and $\gamma:\R\to \ell$ and $N:\R\to \gamma^*TM$ are, respectively, a parametrization of $\ell$ and a normal vector field along $\gamma$ such that $|\gamma'| = |N| = 1$, then $\nabla_{d/dt} \gamma' = \nabla_{d/dt} N = 0$.
    The pair $(\Zoll,g)$ uniquely defines the connection $\nabla$ in this way, and the geodesics of $\nabla$ are the curves $\ell\in\Zoll$.
    Therefore, $[\nabla]$ is a Zoll projective structure.

    We also remark that Theorems A and B hold for $n+1=2$ by \cite[Section 2]{LM}, but the proof of Proposition \ref{prop: Sigma is diffeomorphic to S^n through psi_sigma} follows from the fact that any surface equipped with a Zoll projective structure is either a sphere or a projective plane.
    \hfill $\lozenge$
    \label{rmk: the case n+1=2}
\end{remark}

We are now ready to prove Theorem A under the assumption that the elements of the transitive family are two-sided.

\begin{theorem}
    Let $M$ be an $(n+1)$-dimensional manifold equipped with an unoriented transitive family $\Zoll = \{\Sigma_\sigma:\sigma\in P\}$ of closed embedded hypersurfaces.
    If the elements of $\Zoll$ are two-sided, then $M$ is diffeomorphic to $\Sph^{n+1}$ and the hypersurfaces in $\Zoll$ are diffeomorphic to $\Sph^n$.
    Moreover, if we identify the points $p\in\Sph^n$ with $(p,0)\in\Sph^n\times\{0\}\subset \Sph^{n+1}$, the map $\psi_\sigma:\Sigma_\sigma\to\Sph^n$ in Proposition \ref{prop: Sigma is diffeomorphic to S^n through psi_sigma} extends to a diffeomorphism from $M$ to $\Sph^{n+1}$.
    \label{thm: classification of transitive families of two-sided hypersurfaces}
\end{theorem}
\begin{proof}
    The proof is divided in three parts.
    In Part I, we prove that $M$ is diffeomorphic to $\Sph^{n+1}$ under the assumption that each hypersurface $\Sigma\in\Zoll$ separates $M$ in two components; and in Part II, we prove that if the elements of $\Zoll$ are two-sided, they are separating.
    Together, Parts I and II imply the first assertion of the theorem.

    Part III proves the last assertion about $\psi_\sigma$.
    This is a refinement of Part I, and will only be used when we deal with unoriented transitive families of closed one-sided embedded hypersurfaces.
    We mention that, a priori, $\psi_\sigma$ could have no extension to a diffeomorphism from $M$ to $\Sph^{n+1}$, even though $M$ is diffeomorphic to $\Sph^{n+1}$.
    This is related to the fact that there exist diffeomorphisms of $\Sph^n$ that do not extend to diffeomorphisms of the closed $(n+1)$-ball (see \cite{Mil} and also \cite{Wall}).

    Throughout the proof, we take some Riemannian metric $g$ on $M$.
    Whenever we fix a unit normal vector field $N$ along a hypersurface $\Sigma_\sigma\in\Zoll$, we write $\overline{\zeta} = \phi_\zeta\cdot N$ for $\zeta\in T_\sigma P$ as in the proof of Proposition \ref{prop: Zoll dimension reduction}.
    As an abuse of notation, we also call the functions $\phi_\zeta\in C^\infty(\Sigma_\sigma)$, $\zeta\in T_\sigma P$, the \textit{infinitesimal variations of $\Sigma_\sigma$ along $\Zoll$}.
    By Proposition \ref{prop: Zoll dimension reduction}, a function $\phi\in\{\phi_\zeta: \zeta\in T_\sigma P\}$ is uniquely determined by its value $\phi(x)$ and its gradient $\nabla\phi(x)$ at any point $x\in\Sigma_\sigma$.

    \begin{namedtheorem}[Part I]
        $M$ is diffeomorphic to $\Sph^{n+1}$ if each $\Sigma_\sigma$ is separating.
    \end{namedtheorem}
    \textit{Step 1.} $\Sigma_\sigma$ separates $M$ in two $(n+1)$-balls.
    
    Fix $\Sigma\in\Zoll$.
    By our assumption, $\Sigma$ divides $M$ in two components, say $U_0$ and $U_1$, and we let $\Omega = \overline{U_0}$.
    Observe that, a priori, $\Omega$ may not be compact, since we do not assume compactness of $M$.
    We claim that $\Omega$ is diffeomorphic to the closed unit ball $\bar{B}^{n+1}\subset\R^{n+1}$.
    Suppose that this is true.
    Then we may replace $U_0$ with $U_1$ to conclude that $M$ is obtained from gluing two $(n+1)$-dimensional balls along its boundary spheres through a diffeomorphism.
    In particular, this already implies that $M$ is homeomorphic to $\Sph^{n+1}$.

    Now we turn to the proof of our claim.
    Fix $p\in \Sigma$ and let $S\subset \Sigma$ be a nodal set that passes through $p$.
    Let $v\in T_p\Sigma$ and $N$ be, respectively, a choice of unit vector perpendicular to $T_pS$ and a unit normal vector field along $\Sigma$ that points towards $\Omega$.
    For $t\in [0,1]$, we let $\Sigma_t\in\Zoll$ be the unique hypersurface that passes through $p$ with $T_p\Sigma_t$ equal to $T_pS\oplus \langle v_t\rangle$, where $v_t = \cos(\pi t) v + \sin(\pi t)N_p$.
    Observe that both $\Sigma_0 = \Sigma_1 = \Sigma$.

    By the properties of $\Zoll$, $\Sigma_t$ intersects $\Sigma = \Sigma_0$ transversely for $t\in(0,1)$.
    In particular, $S_t := \Sigma_t\cap \Sigma$ is a smooth submanifold of $M$ of codimension two, and $\Sigma_t^+:= \Sigma_t\cap\Omega$ is a smooth compact hypersurface with boundary equal to $S_t$.
    Moreover, since $\Sigma_t$ varies smoothly in $t$ in the $C^\infty$ topology, the same is true for $S_t$ and for $\Sigma_t^+$.

    We also observe two things.
    First, each $S_t$ is diffeomorphic to $\Sph^{n-1}$, since $S_t$ converges in the $C^\infty$ topology to the nodal set $S$ as $t\to 0$ (and as $t\to 1$) and, from Proposition \ref{prop: Sigma is diffeomorphic to S^n through psi_sigma}, there is a diffeomorphism from $\Sigma$ to $\Sph^n$ that takes $S$ to an equator.
    This justifies the notation $S_0 := S_1 := S$.
    Second, $\Sigma\cup\Sigma_t$ divides $M$ in exactly four components whenever $t\in(0,1)$.

    Since the application $t\in[0,1]\mapsto \Sigma_t\in\Zoll$ is smooth in the $C^\infty$ topology, we can write $\Sigma_t$ as a normal graph over $\Sigma_s$ whenever $t$ and $s$ are sufficiently close.
    Therefore, there exists an isotopy $f:\Sigma\times[0,1]\to M$ such that $\Sigma_t = f(\Sigma,t)$, $f(x,0) = x$ for every $x\in\Sigma$ and $f(p,t) = p$ for every $t\in [0,1]$.
    By the isotopy extension theorem, we can extend $f$ to an isotopy $F:M\times [0,1]\to M$ starting from the identity map $F(x,0) = x$ for all $x\in M$.
    
    Here, the isotopy extension theorem can be applied, because $\Sigma$ is compact.
    We also remark that the extension $F$ is constructed to have compact support, in the sense that $F(x,t) = x$ for all pairs $(x,t)\in M\times[0,1]$ outside a compact subset (see \cite[Theorem 2.4.2]{Wall}).

    Our choice of normal $N$ along $\Sigma$ can be smoothly deformed to give a time-dependent normal vector field $N^t$ along each $\Sigma_t$.
    By construction, $N^t_p = \cos(\pi t)N_p -\sin(\pi t)v$ and, hence, $N^0 = N = -N^1$.
    Thus, the isotopy $F$ induces an isotopy from $U_0$ to $U_1$, meaning that $F(U_0,1) = U_1$.

    Writing $U_t:= F(U_0,t)$ and $\Omega_t:= \Omega\cap\overline{U_t} = \overline{U_0}\cap \overline{U_t}$, we see that $\Omega_0 = \Omega$, $\Omega_1 = \Sigma$ and, for $t\in(0,1)$, $\Omega_t$ has the structure of a smooth manifold with corner whose boundary is $\partial \Omega_t:= \Sigma^+_t \cup \Sigma^-_t$, where $\Sigma^+_t = \Sigma_t\cap \overline{U_0}$ and $\Sigma^-_t = \Sigma_0\cap \overline{U_t}$.
    The corner of $\Omega_t$ is $S_t = \partial\Sigma^+_t = \partial\Sigma^-_t$ and is two-sided with sides $\Sigma^\pm_t$.
    The same argument we used to obtain the smoothness of $t\mapsto \Sigma^+_t$ in the $C^\infty$ topology gives the smooth dependence on $t$ of $\Sigma^-_t$ in the $C^\infty$ topology.

    For any $r>0$ and any subset $A\subset M$, write $B_r(A) = \{x\in M:d_g(x,A)<r\}$.
    Let $\eta>0$ be sufficiently small to ensure that $B_\eta(\Sigma)$ is a tubular neighborhoods of $\Sigma$.

    Let $\epsilon>0$ be small enough so that $\Sigma_t\in B_{\eta/3}(\Sigma)$ whenever $t\leq\epsilon$ or $t\geq1-\epsilon$.
    Taking $\epsilon$ smaller, if necessary, we may assume that $\Sigma_t$ is the graph of a function $u_t$ over $\Sigma$ with respect to the unit normal vector field $N$.
    In other words,
    \[ \Sigma_t =\{\exp_x(u_t(x)N_x):x\in \Sigma\} \]
    whenever $t\leq \epsilon$ or $t\geq 1-\epsilon$.
    Here the exponential map is taken with respect to $g$.

    By the compactness of the interval $[\epsilon,1-\epsilon]$, by the smooth dependence on $t$ of $\Sigma^+_t$, $\Sigma^-_t$ and $S_t$ with respect to the $C^\infty$ topology, and since $S_t$ is a two-sided corner, there exists a smooth map $G:S_\epsilon\times [0,1]^2\times[\epsilon,1-\epsilon]\to M$ such that, for every $t\in[\epsilon,1-\epsilon]$, $G_t:=G(\cdot,\cdot,t):S_\epsilon\times[0,1]^2\to M$ is a diffeomorphism onto a collar neighborhood of $S_t$ inside $\widetilde{\Omega}_t$ (see \cite[Lemma 2.6.1]{Wall}).
    The map $G_t$ is taken to identify $S_\epsilon\times\{(1,0)\}$ with $S_t$, $S_\epsilon\times[0,1]\times\{0\}$ with a collar neighborhood of $S_t$ inside $\Sigma^-_t$, and $S_\epsilon\times\{0\}\times[0,1]$ with a collar neighborhood of $S_t$ inside $\Sigma^+_t$ (see the left image in Figure 1 below).

    Take $\delta\in(0,1)$ sufficiently small, and define a smooth curve inside the square $[0,1]^2$, connecting the points $(1-\delta,0)$ and $(1,\delta)$.
    This curve can be taken as the graph of an increasing smooth convex function $\rho:[1-\delta,1]\to[0,\delta]$, and we define $A:=\{(s,r)\in[1-\delta,1]\times [0,\delta]:r<\rho(s)\}$, an open subset relative to $[0,1]^2$.
    
    We straighten the corner $S_t\subset \Omega_t$ taking $\widetilde{\Omega}_t := \Omega_t\setminus G_t(S_\epsilon\times A)$ for $t\in[\epsilon,1-\epsilon]$ (see the right image in Figure 1 below).
    Then $\widetilde{\Omega}_t$ is contained in $\Omega_t$, is homeomorphic to $\Omega_t$ and has the structure of a smooth manifold with boundary, but without corners.

    \begin{figure}[ht]
    \centering
    \begin{tikzpicture}[scale=2.5]
    \begin{scope}[shift={(-1.8,0)}]
        \fill[blue!8] (0,0) rectangle (1,1);
        
        \draw[thick] (0,0) -- (1,0) -- (1,1);
        
        \fill (1,0) circle (.8pt);
        
        \node[below right] at (1,0) {$S_t$};
        \node[below] at (0.5,0) {$\Sigma^-_t$};
        \node[right] at (1,0.5) {$\Sigma^+_t$};
        
        \node at (0.5,0.5) {$\Omega_t$};
    \end{scope}
    \draw[->, thick, >=stealth] (-0.4,0.5) -- (0.1,0.5);
    \node[above] at (-0.15,0.5) {\small smoothing};
    \begin{scope}[shift={(0.5,0)}]
        \def\eps{0.25}
        
        \fill[blue!8] 
            (0,0) -- (0.75,0) .. controls (0.92,0) and (1,0.08) .. (1,\eps) -- (1,1) -- (0,1) -- (0,0) -- cycle;
        
        \draw[thick] (1,\eps) -- (1,1);
        \draw[thick] (0,0) -- (0.75,0);
        
        \draw[thick, red!70!black] (0.75,0) .. controls (0.92,0) and (1,0.08) .. (1,\eps);

        \draw[dashed, gray] (0.75,0) -- (1,0);
        \draw[dashed, gray] (1,0) -- (1,\eps);
        
        \node[below] at (0.75,0) {\small $(1-\delta,0)$};
        \node[right] at (1,\eps) {\small $(1,\delta)$};
        
        \node at (0.5,0.5) {$\widetilde{\Omega}_t$};
    \end{scope}
    \end{tikzpicture}
    \caption{Smoothing a codimension-two corner. Left: the collar neighborhood 
    $S_t \times [0,1]^2$ with corner at $S_t \times \{(1,0)\}$. Right: the smoothed 
    domain $\widetilde{\Omega}_t$ obtained by replacing the corner with a smooth convex curve connecting $(1-\delta,0)$ to $(1,\delta)$. This image was generated by the free online version of Anthropic's Claude.}
    \end{figure}
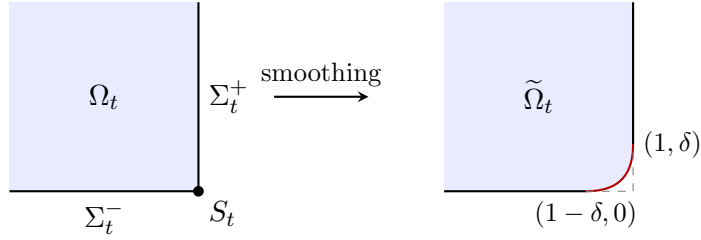

    By construction, the boundary $\partial\widetilde{\Omega}_t$ varies smoothly in $t$ in the $C^\infty$ topology.
    Moreover, for any $\eta'\in(0,\eta/3)$, we can take $\delta = \delta(\eta')>0$ sufficiently small so that the symmetric difference $\Omega_t\triangle \widetilde{\Omega}_t = \Omega_t\setminus\widetilde\Omega_t$ has closure contained in the tubular neighborhood $B_{\eta'}(S_t)$.

    Since the dependence on $t\in[\epsilon,1-\epsilon]$ of the boundary $\partial\widetilde\Omega_t$ is smooth in the $C^\infty$ topology, we can write $\partial\widetilde\Omega_s$ as a graph over $\partial\widetilde\Omega_t$ whenever $s$ and $t$ are sufficiently close to each other.
    Therefore, by the same argument as we did with $\Sigma_t$, we can construct an isotopy $\widetilde{f}:\partial\widetilde\Omega_\epsilon\times[\epsilon,1-\epsilon]\to M$ such that $\widetilde{f}(\partial\widetilde{\Omega}_\epsilon,t) = \partial\widetilde{\Omega}_t$ for all $t$ and $\widetilde{f}(x,\epsilon) = x$ for every $x\in\partial\widetilde{\Omega}_\epsilon$.
    
    Once again, by the isotopy extension theorem, we extend $\widetilde{f}$ to obtain an isotopy $\widetilde{F}:M\times[\epsilon,1-\epsilon]\to M$ such that $\widetilde{F}(x,\epsilon) = x$ for all $x\in M$.
    By construction, $\widetilde{F}(\widetilde{\Omega}_\epsilon,t) = \widetilde{\Omega}_t$ for all $t\in[\epsilon,1-\epsilon]$.
    In particular, $\widetilde{\Omega}_\epsilon$ is diffeomorphic to $\widetilde{\Omega}_{1-\epsilon}$.

    We claim that $\Omega$ is isotopic to $\widetilde{\Omega}_\epsilon$, and that $\widetilde{\Omega}_{1-\epsilon}$ is diffeomorphic to the closed unit ball $\bar{B}^{n+1}\subset\R^{n+1}$.
    This will finish our first step.

    For the first statement (that $\Omega$ is isotopic to $\widetilde{\Omega}_\epsilon$), we note that, taking $\delta>0$ sufficiently small, the boundary $\partial\widetilde{\Omega}_\epsilon$ is the normal graph over $\Sigma$.
    In other words, there is some nonnegative function $\widetilde{u}_\epsilon\in C^\infty(\Sigma)$ such that
    \[ \partial\widetilde{\Omega}_\epsilon = \{ \exp_x(\widetilde{u}_\epsilon(x)N_x): x\in\Sigma \}. \]
    Since $\partial\widetilde{\Omega}_\epsilon \subset B_{\eta/3}(\Sigma)\cap \Omega$, and since $N$ points towards $\Omega$, we have the bound $0\leq\widetilde{u}_\epsilon\leq \eta/3$.
    Now, take a non-increasing smooth function $h:\R\to[0,1]$ such $h(r) = 1$ for $r<\eta/3$ and $h(r) = 0$ for $r>2\eta/3$.
    Then, define the isotopy $\Omega\times[0,1]\to\Omega$ that maps $(x,t)$ to $x$ if $x\notin B_\eta(\Sigma)$ and takes $(\exp_y(rN_y),t)$ to $\exp_y((r+th(r)\widetilde{u}_\epsilon(y))N_y)$ for all $y\in\Sigma$ and $r\in(0,\eta)$.
    Since $h(r)=0$ for $r>2\eta/3$, this defines a smooth map $\Omega\times[0,1]\to \Omega$ that is, by construction, an embedding when restricted to each $\Omega\times\{t\}$.
    Moreover, it maps $\Omega\times\{1\}$ onto $\widetilde{\Omega}_\epsilon$, and this proves the first claim.

    We now prove that $\widetilde{\Omega}_{1-\epsilon}$ is diffeomorphic to $\bar{B}^{n+1}$.
    First, we observe that the set $\Omega_{1-\epsilon}$ is contained in the tubular neighborhood $B_{\eta/3}(\Sigma)$.
    In fact, since $\Sigma_{1-\epsilon}$ is the graph over $\Sigma$ of the function $u_{1-\epsilon}\in C^\infty(\Sigma)$, since $\Omega_{1-\epsilon}\subset\Omega$, and since $N$ points towards $\Omega$, $u_{1-\epsilon}\geq 0$ on $\Sigma^-_{1-\epsilon}$ and vanishes precisely at $S_{1-\epsilon} = \partial\Sigma^-_{1-\epsilon}$.
    Therefore, our construction of $\Omega_{1-\epsilon}$ ensures that $\Omega_{1-\epsilon} = \{\exp_y(rN_x):x\in\Sigma^-_{1-\epsilon}, 0\leq r\leq u_{1-\epsilon}(x)\}$.
    This is because a tangent vector $w\in T_xM$ at $x\in S_t$ points to the interior of $\Omega_t$ if and only if both $g(N^t_x,w)>0$ and $g(N_x,w)>0$.
    As a consequence, $\widetilde{\Omega}_{1-\epsilon}$ is also contained in $B_{\eta/3}(\Sigma)$.
    
    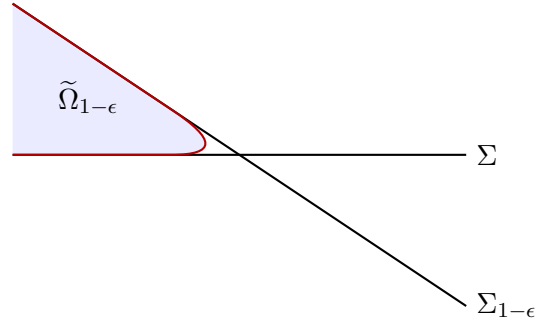
\begin{figure}[ht]
        \centering
        \begin{tikzpicture}
            \fill[blue!8] (-3,0) -- (-0.9,0) .. controls (-0.3,0) and (-0.3,0.2) .. (-0.9,0.6) -- (-3,2) -- cycle;
            
            \draw[thick] (-3,0) -- (3,0);
            \draw[thick] (-3,2) -- (3,-2);

            \draw[thick, red!70!black] (-3,0) -- (-0.9,0) .. controls (-0.3,0) and (-0.3,0.2) .. (-0.9,0.6) -- (-3,2);
            
            \node[right] at (3,0) {$\Sigma$};
            \node[right] at (3,-2) {$\Sigma_{1-\epsilon}$};
            \node at (-2,.75) {$\widetilde{\Omega}_{1-\epsilon}$};
        \end{tikzpicture}
        \caption{The domain $\widetilde\Omega_{1-\epsilon}$ is contained in the tubular neighborhood $B_{\eta/3}(\Sigma)$. This image was partially generated by the free online version of Anthropic's Claude.}
        \label{fig:placeholder}
    \end{figure}

    Since $S_{1-\epsilon}$ is isotopic, inside $\Sigma$, to $S_1 = S$, $\Sigma^-_{1-\epsilon}$ is diffeomorphic to the closed $n$-dimensional ball $\bar{B}^n\subset\R^n$.
    The same is true for $\Sigma^+_{1-\epsilon}$, because it is the graph of $u_{1-\epsilon}$ over $\Sigma^-_{1-\epsilon}$.
    Since $u_{1-\epsilon}>0$ in the interior of $\Sigma^-_{1-\epsilon}$, our description of $\Omega_{1-\epsilon}$ as the region under the graph of the restriction of $u_{1-\epsilon}$ to $\Sigma^-_{1-\epsilon}$ already implies that $\Omega_{1-\epsilon}$ is homeomorphic to the closed unit $(n+1)$-ball.
    In fact, since $\nabla u_{1-\epsilon}$ never vanishes at any point of $S_{1-\epsilon} = \{u_{1-\epsilon}=0\}$, $\Omega_{1-\epsilon}$ is diffeomorphic to the closed unit $(n+1)$-ball with a corner introduced at an equator of $\Sph^n$.
    Since $\widetilde{\Omega}_{1-\epsilon}$ is obtained from $\Omega_{1-\epsilon}$ by straightening the corner $S_{1-\epsilon}$, it is diffeomorphic to $\bar{B}^{n+1}$ (see \cite[Proposition 2.6.2, Lemma 2.6.3 and Proposition 2.6.3]{Wall}).
    
    \textit{Step 2.} $M$ is diffeomorphic to $\Sph^{n+1}$.
    
    We proved that $M$ is the union of two diffeomorphic balls $\Omega$ and $F_1(\Omega) = F(\Omega\times\{1\})$ along the identification $f_1= f(\cdot,1):\Sigma\to\Sigma$ of their common boundary $\Sigma = \partial \Omega = \partial F_1(\Omega)$.
    To show that it is diffeomorphic to $\Sph^{n+1}$, it suffices to prove that, after identifying $\Sigma=\Sigma_0$ with $\Sph^{n}$, $f_1$ is isotopic to $j:(x_0,...,x_n)\in \Sph^n\mapsto (x_0,...,x_{n-1},-x_n)\in \Sph^n$ (see \cite{Mil}).

    Let $\sigma_t\in P$ be the smooth path induced by $\Sigma_t = \Sigma_{\sigma_t}$.
    For $t\in[0,1]$, let $\zeta\in T_{\sigma_t}P\mapsto \phi_\zeta\in C^\infty(\Sigma_t)$ be the injective homomorphism that assigns $\zeta$ to the function $\phi_\zeta = \phi^t_\zeta$ that is the scalar part of its induced variation $\phi_\zeta N^t$ along $\Sigma_t$.
    For $t\in[0,1]$ and $x\in \Sigma_t$, we write $E_{x.t}:\zeta\in T_x\Sigma_t\mapsto (\phi_\zeta(x),\nabla\phi_\zeta(x))\in \R\times T_x\Sigma_t$ for the corresponding isomorphism described in \eqref{map: the isomorphism E^g_x,sigma}.
    
    We observe that $\phi^1_\zeta = -\phi^0_\zeta$ for all $\zeta\in T_{\sigma_0}P$.
    This happens because $\Sigma_0 = \Sigma_1 = \Sigma$ and $N^0 = - N^1$, so we must have the equality $\phi^0_\zeta N^0 = \phi^1_\zeta N^1 = - \phi^1_\zeta N^0$ for all $\zeta\in T_{\sigma_0}P$.
    Consequently, $E_{x,1} = - E_{x,0}$ for all $x\in \Sigma$.
    
    By construction, $\Sigma_t$ contains the point $p$ and $T_p S\subset T_p\Sigma_t$ for every $t\in[0,1]$.
    Take an orthonormal basis $w_1,...,w_{n-1}\in T_pS$, and let $\zeta^t_0,...,\zeta^t_n\in T_{\sigma_t} P$ be the unique basis whose associated variations $\phi^t_i:= \phi_{\zeta^t_i}\in C^\infty(\Sigma_t)$ satisfy:
    \[
        \begin{split}
            \phi_0^t(p) = 1, \ & \nabla\phi_0^t(p) = 0 \\
            \phi_i^t(p) = 0, \ & \nabla\phi_i^t(p) = w_i \ \ \   (i=1,...,n-1) \\
            \phi_n^t(p) = 0, \ & \nabla\phi_n^t(p)=v_t
        \end{split}
    \]
    for all $t\in[0,1]$.
    We observe that the relation $E_{x,1} = - E_{x,0}$ for all $x\in\Sigma$ implies that $\phi_i^0 = \phi_i^1$ for all $i=0,...,n-1$ and $\phi_n^1 = - \phi_n^0$, because $v_1 = - v_0$.

    Define $\Gamma:= \{(x,t)\in M\times[0,1] : x\in\Sigma_t\}$.
    Since $\Sigma_t$ depends smoothly on $t$ in the $C^\infty$ topology, $\Gamma$ is a smooth embedded hypersurface of $M\times[0,1]$.
    By Proposition \ref{prop: Sigma is diffeomorphic to S^n through psi_sigma}, the maps $\widetilde{\psi}_t(x):=(\phi^t_0(x),...,\phi^t_n(x))\in\R^{n+1}$ induce diffeomorphisms $\psi_t = \widetilde{\psi}_t/|\widetilde{\psi}_t|:\Sigma_t\to \Sph^n$ that vary smoothly in $t$, in the sense that $(x,t)\in\Gamma\mapsto (\psi^t(x),t)\in \Sph^n\times[0,1]$ is a diffeomorphism.

    Write $\hat{f}_1 = \psi_0\circ f_1\circ \psi_0^{-1}:\Sph^n\to \Sph^n$.
    Our goal is to prove that $\hat{f}_1$ is isotopic to $j$.
    Since $f_t = f(\cdot,t):\Sigma\to M$ maps $\Sigma = \Sigma_0$ diffeomorphically to $\Sigma_t$, we can define the isotopy $H:(y,t)\in \Sph^n\mapsto j\circ\psi_t\circ f_t\circ \psi_0^{-1}(y)\in \Sph^n$.
    Note that $H(y,0) = i(y)$ for all $y\in \Sph^n$, for $f_0$ is the identity map.
    Moreover, $H(\cdot,1) = j\circ\psi_1\circ f_1\circ \psi_0^{-1} = j\circ j\circ \psi_0\circ f_1\circ\psi_0 = \psi_0\circ f_1\circ\psi_0^{-1} = \hat{f}_1$, because $\psi_1 = j\circ \psi_0$.

    Therefore, $H$ is an isotopy between $j$ and $\hat{f}_1$, as we claimed.
    This shows that there is a diffeomorphism from $M$ to $\Sph^{n+1}$ that sends $\Sigma$ to an equator.

    \begin{namedtheorem}[Part II]
        If the hypersurfaces in $\Zoll$ are two-sided, they are separating.
    \end{namedtheorem}

    Suppose that each $\Sigma_\sigma\in\Zoll$ is two-sided.
    Let $\Sigma\in\Zoll$, $p\in\Sigma$, $S\subset \Sigma$ and $v\in T_p\Sigma$ be as in the proof of Part I.
    Since $\Sigma$ is two-sided, there is a unit normal vector field $N$ along $\Sigma$.
    Therefore, we can define $v_t$, $\Sigma_t$, $S_t$ and $N^t$ in an analogous way as we did in Part I.

    Observe that $S_t$ divides $\Sigma$ in two components, both diffeomorphic to open $n$-balls.
    We let $\Sigma^-_t$ be the closure of the component of $\Sigma\setminus S_t$ for which $-v$ is pointing towards.
    The analogous is true for $\Sigma_t$: $S_t$ cuts $\Sigma_t$ along two diffeomorphic open $n$-balls, and we denote by $\Sigma^+_t$ the closure of the one which $v_t$ is pointing towards.

    We take $\epsilon>0$ as before, so that $\Sigma_t$ is the normal graph over $\Sigma$ whenever $t\in[0,\epsilon]\cup[1-\epsilon,1]$.
    Since $[\epsilon,1-\epsilon]$ is compact, and since $S_t$ is the transversal intersection of $\Sigma_t$ with $\Sigma$, there is a smooth map $G:S\times [-1,1]^2\times [\epsilon,1-\epsilon]\to M$ so that, for each $t\in[\epsilon,1-\epsilon]$, $G_t = G(\cdot,\cdot,t):S\times [-1,1]^2\to M$ is a diffeomorphism from $S\times[-1,1]^2$ to a tubular neighborhood of $S_t$.
    Moreover, we can take $G$ so that $G_t(S\times [-1,1]\times\{0\})\subset \Sigma$ and $G_t(S\times\{0\}\times [-1,1])\subset \Sigma_t$.
    We may also assume that $G_t(S\times [-1,0]\times\{0\})\subset \Sigma^-_t$ and $G_t(S\times\{0\}\times [0,1])\subset \Sigma^+_t$.

    We then do the same smoothing procedure described in the proof of Part I to define a smooth embedded hypersurface $\widetilde{\Sigma}_t$ that is isotopic to $\Sigma$ and is obtained by smoothing the corner $S_t$ of $\Sigma^-_t\cup \Sigma^+_t$.
    We do this procedure for $t\in[\epsilon,1-\epsilon]$, and note that, since $\Sigma^+_{1-\epsilon}$ is the graph over $\Sigma^-_{1-\epsilon}$, $\Sigma^+_{1-\epsilon}\cup\Sigma^-_{1-\epsilon}$ encloses a diffeomorphic $(n+1)$-ball with a corner introduced at the equator.
    Since $\widetilde{\Sigma}_{1-\epsilon}$ is obtained from $\Sigma^+_{1-\epsilon}\cup\Sigma^-_{1-\epsilon}$ by smoothing along $S_{1-\epsilon}$, it encloses a region diffeomorphic to the $(n+1)$-ball.
    Therefore, $\widetilde{\Sigma}_{1-\epsilon}$ is separating, and the same is true for $\Sigma$, for it is isotopic to $\widetilde{\Sigma}_{1-\epsilon}$.

    \begin{namedtheorem}[Part III]
        $\psi_\sigma$ extends to a diffeomorphism $M\to\Sph^{n+1}$.
    \end{namedtheorem}

    Now that we know that all hypersurfaces in $\Zoll$ are separating, we go back to the notation in part one and consider $\Sigma$, $\Sigma_t$, $S$, $S_t$, $p$, $N$, $N^t$, $v$, $v_t$, $\Omega$ and $\Omega_t$ as defined there.

    \textit{Step 3.} Some preliminary constructions.
    
    Observe that the set $\Gamma':=\{(x,t)\in \Sigma\times[0,1]:x\in S_t\}$ is a submanifold of $\Sigma\times[0,1]$ diffeomorphic to $S\times [0,1]$ through an embedding of the form $(x,t)\in\Gamma'\mapsto(\gamma(x,t),t)\in \Sigma\times[0,1]$, where $\gamma(x,0) = x$ for all $x\in S$ and $\gamma(p,t) = p$ for all $t\in[0,1]$, that maps $\Gamma'$ to $S\times[0,1]$ and is isotopic to the inclusion $\Gamma'\subset \Sigma\times[0,1]$.
    This is true by the smooth dependence on $t\in[0,1]$ of $S_t$ in the $C^\infty$-topology, and by the fact that every $S_t$ contains $p$.

    By the isotopy extension theorem, we extend $(x,t)\mapsto(\gamma(x,t),t)$ to a diffeomorphism $\overline{\gamma}:M\times[0,1]\to M\times[0,1]$ that is isotopic to the identity $id:M\times[0,1]\to M\times[0,1]$, maps $\Sigma\times [0,1]$ to itself and restricts to the inclusion on $M\times\{0\}$.
    Moreover, $\overline{\gamma}$ can be taken so that the composition of $\overline{\gamma}$ with the projection $(x,t)\in M\times[0,1]\mapsto t\in[0,1]$ is the projection itself, and we write $\overline{\gamma}(x,t) = (\gamma(x,t),t)$ as an abuse of notation, so that $\gamma:M\times[0,1]\to M$ is an isotopy with $\gamma(\cdot,0) = id_M$.

    Replacing $\Sigma_t$ by $\gamma(\Sigma_t\times\{t\})$, we may assume that $\Sigma_t\cap \Sigma = S$ for all $t\in(0,1)$.
    Moreover, the collection $\{\Sigma_t\}_{t\in[0,1]}$ constitutes a smooth one-parameter family of hypersurfaces in the $C^\infty$ topology, and each $\Sigma_t$ is an element of the transitive family $\Zoll_t:=\{\gamma(\Sigma'\times\{t\}):\Sigma'\in \Zoll\}$.

    Observe also that $\Sigma_1 = \gamma(\Sigma\times\{1\}) = \gamma(\Sigma\times\{t\}) = \Sigma$ for all $t\in[0,1]$.
    In particular, $\Sigma\in\Zoll_t$ for all $t\in[0,1]$.

    Take a $g$-orthonormal basis $w_1,...,w_{n-1}\in T_pS_0$, and let $\phi_0^t,...,\phi_n^t\in C^\infty(\Sigma_t)$ be the unique infinitesimal variations of $\Sigma_t$ along $\Zoll_t$ with the prescribed conditions
    \[
        \begin{split}
            \phi_0^t(p) = 1, \ & \nabla\phi_0^t(p) = 0 \\
            \phi_i^t(p) = 0, \ & \nabla\phi_i^t(p) = w_i \ \ \   (i=1,...,n-1) \\
            \phi_n^t(p) = 0, \ & \nabla\phi_n^t(p)=v_t.
        \end{split}
    \]
    Recall that the existence and uniqueness of $\phi^t_0,...,\phi^t_n$ follows from Proposition \ref{prop: Zoll dimension reduction}.
    For each $t\in[0,1]$, we let $\psi_t:\Sigma_t\to \Sph^n$ be the diffeomorphism constructed as in Proposition \ref{prop: Sigma is diffeomorphic to S^n through psi_sigma}, given by
    \[ \psi_t(x) = \frac{(\phi^t_0(x),...,\phi^t_n(x))}{|(\phi^t_0(x),...,\phi^t_n(x))|}. \]
    
    Note that $\psi_t$ varies smoothly in $t$ in the sense that the set $\Gamma:=\{(x,t)\in M\times[0,1]:x\in \Sigma_t\}$ is a smooth hypersurface of $M\times[0,1]$ diffeomorphic to $\Sph^n\times[0,1]$ by the map $(x,t)\in\Gamma\mapsto(\psi_t(x),t)\in \Sph^n\times[0,1]$.
    Moreover, since $\gamma(\cdot,0)=id$ on $M$, we see that $\Zoll_0 = \Zoll$.

    Since $\Sigma_t\cap \Sigma = S$ does not depend on $t\in(0,1)$, the family $\psi_t$ induces an isotopy $(x,t)\in S\times [0,1]\mapsto \psi_t(x)\in \Sph^n$.
    We also observe that $\psi_0$ embeds $S$ in $\Sph^n$ as the equator $\Sph^{n-1}\times\{0\}$.
    The isotopy extension theorem thus gives an isotopy $\Psi = (\Psi_t)_{t\in[0,1]}:\Sph^n\times[0,1]\to \Sph^n$ such that $\Psi_0 = id$ and $\Psi(\psi_0(x),t) = \psi_t(x)$ for all $t\in[0,1]$.

    If we substitute $\psi_t$ by $\Psi_t^{-1}\circ\psi_t$, we can assume that the diffeomorphism $(x,t)\in\Gamma\mapsto (\psi_t(x),t)\in \Sph^n\times[0,1]$ restricts to the map $(x,t)\in S\times[0,1]\mapsto (\psi_0(x),t)\in \Sph^n\times[0,1]$.
    In other words, $\psi_t$ is equal to $\psi_0$ on $S$ and maps it to the equator $\Sph^{n-1}\times\{0\}\subset \Sph^n$.

    We do a similar procedure for $\Sigma$.
    Since $\Sigma\in\Zoll_t$ for all $t$, we may define $\overline{\phi}_0^t,...,\overline{\phi}_n^t\in C^\infty(\Sigma_0)$ as the unique infinitesimal deformations of $\Sigma$ along $\Zoll_t$ with first order conditions
    \[
        \begin{split}
            \overline{\phi}_0^t(p) = 1, \ & \nabla\overline{\phi}_0^t(p) = 0 \\
            \overline{\phi}_i^t(p) = 0, \ & \nabla\overline{\phi}_i^t(p) = w_i \ \ \   (i=1,...,n-1) \\
            \overline{\phi}_n^t(p) = 0, \ & \nabla\overline{\phi}_n^t(p)=v_0.
        \end{split}
    \]
    We also define the smooth one-parameter family of diffeomorphisms $\varphi_t:\Sigma\to \Sph^n$ as
    
    \[ \varphi_t(x) = \frac{\left(\overline{\phi}^t_0(x),...,\overline{\phi}^t_n(x)\right)}{\left|\left(\overline{\phi}^t_0(x),...,\overline{\phi}^t_n(x)\right)\right|}. \]

    We observe that $\varphi_0 = \psi_0$ and that $\varphi_1 = j\circ \psi_1$, where $j:\Sph^n\to \Sph^n$ is the inversion $j(y_0,...,y_n) = (y_0,...,y_{n-1},-y_n)$.
    This second statement holds, because $\Sigma_1 = \Sigma_0 =\Sigma$ and $v = v_0 = -v_1$.

    The map $(x,t)\in S\times[0,1]\mapsto\varphi_t(x)\in \Sph^n$ is an isotopy that sends $S\times\{0\}$ to the equator $\Sph^{n-1}\times\{0\}\subset \Sph^n$.
    It also fixes $p\in \Sigma$ to the first order, in the sense that both $\varphi_t(p) = \varphi_0(p)$ and $D\varphi_t(p) = D\varphi_0(p)$ for all $t\in[0,1]$.
    
    Therefore, the isotopy extension theorem give us an isotopy $\Phi = (\Phi_t)_{t}:\Sph^n\times[0,1]\to \Sph^n$ such that $\varphi_t(x) = \Phi_t(\varphi_0(x))$ for all $x\in S_0$.
    Moreover, $\Phi_t(\varphi_0(p)) = \varphi_0(p)$ and $D\Phi_t(\varphi_0(p)) = D\Phi_0(\varphi_0(p))$ for all $t\in[0,1]$.

    Replacing $\varphi_t$ by $\Phi_t^{-1}\circ\varphi_t$, we may assume that $\varphi_t(S) = \Sph^{n-1}\times\{0\}$ for all $t\in[0,1]$.
    Moreover, this change makes $\varphi_t = \varphi_0 = \psi_0 = \psi_t$ on $S$ for all $t$.

    \textit{Step 4.} $\psi_\sigma$ extends to $\Omega$.
    
    In the proof of Part I, we argued on the existence of an isotopy $f:\Sigma\times[0,1]\to M$ such that $f(x,0) = x$ for all $x\in\Sigma_0$ and $f(\Sigma\times\{t\}) = \Sigma_t$ for all $t\in[0,1]$.
    Here, it will be more convenient to identify $\Sigma$ with $\Sph^n$ through $\psi_0$ and take $f = (f_t)_{t\in[0,1]}:\Sph^n\times[0,1]\to M$ explicitly as $f_t = \psi_t^{-1}$.
    We note that $f_t = \psi_0^{-1}$ when restricted the equator $\Sph^{n-1}\times\{0\}$.

    As in the proof of Part I, we write $F:M\times[0,1]\to M$ for a choice of extension of $f$, in the sense that $F(\cdot,0) = id_M$ and $F(f(y,0),t) = f(y,t)$ for all $(y,t)\in \Sph^n\times[0,1]$.
    We also use the notation $\Omega_t = \Omega\cap F(\Omega\times\{t\})$ and note that, for $t\in(0,1)$, $\Omega_t$ has the structure of a smooth manifold with a two-sided corner, with boundary equal to $\partial \Omega_t = \Sigma^-_t\cup \Sigma^+_t$ and corner equal to $\partial \Sigma^-_t = \partial \Sigma^+_t = S$.
    Here, we write $\Sigma^-_t = \Sigma\cap F(\Omega\times\{t\})$ and $\Sigma^+_t = \Sigma_t\cap \Omega$, as we did previously, and observe that our changes in Step 1 assure that $\Sigma^-_t = \Sigma^-_0=:\Sigma^-$ for all $t\in(0,1)$.

    By our construction, $f$ induces a smooth path $h = (h_t)_{t\in[0,1)}:\Sph^n\times[0,1]\to M$ of piecewise smooth embeddings $h_t:\Sph^n\to M$ with image $h_t(\Sigma) = \partial\Omega_t$, and defined as $h_t(y) = \varphi_t^{-1}(y)$ for $y\in B^-:=\{(x_0,...,x_n)\in \Sph^n:x_n\leq 0\}$ and $h_t(y) = f(y,t)$ for $y\notin B^-$.
    Observe that, for $t\in(0,1)$, $h_t$ is smooth when restricted to $\Sigma^-_0$ and to $\Sigma^+_0 := \overline{\Sigma_0\setminus \Sigma^-_0}$, but creates a corner at $S_0$.

    Now, recall that we chose an $\epsilon>0$ sufficiently small and constructed $\widetilde{\Omega}_t\subset \Omega_t$ by smoothing out the corner of $\Omega_t$ for $t\in[\epsilon,1-\epsilon]$.
    This was made in such a way that, for a given $\eta>0$ small enough, the symmetric difference $\widetilde{\Omega}_t\triangle \Omega_t = \Omega_t\setminus\widetilde{\Omega}_t \subset B_\eta(S_0)$.

    Therefore, we can construct an isotopy $\widetilde{f} = (\widetilde{f}_t)_{t\in[\epsilon,1-\epsilon]}:\Sph^n\times[\epsilon,1-\epsilon]\to M$ as follows.
    For $t\in[\epsilon,1-\epsilon]$, $\widetilde{f}_t(y) = h_t(y)$ for $y\in \Sph^n\setminus B_\eta (\Sph^{n-1}\times\{0\})$; on $B_\eta(\Sph^{n-1}\times\{0\})$, $\widetilde{f}_t$ is a small perturbation of $h_t$ in the $C^0$ topology so that to make it smooth at $\Sph^{n-1}\times\{0\}$.
    In fact, we take $\widetilde{f}_t$ to be $h_t$ after smoothing the corner generated by $h_t$ at $\Sph^{n-1}\times\{0\}\subset \Sph^n$.
    Moreover, we do this so to make $\widetilde{f}_t$ a diffeomorphism from $\Sph^n$ to $\widetilde{\Sigma}_t:=\partial\widetilde{\Omega}_t$.

    Note that we can choose $\epsilon$, $\eta$ and $\widetilde{\Sigma}_\epsilon$ so that both $\Sigma_\epsilon$ and $\widetilde{\Sigma}_\epsilon$ are graphical over $\Sigma_0$, $\widetilde{\Sigma}_\epsilon$ is sufficiently close to $\Sigma_0$ in the $C^\infty$ topology, and the function $\widetilde{f}_\epsilon$ can be taken to be sufficiently close to $\psi_0^{-1} = f_0$ in the $C^\infty$ topology.
    Explicitly, for any fixed neighborhood $\mathcal{U}$ of $f_0$ in the $C^\infty$ topology, there are $\epsilon,\eta>0$ and a choice of $\widetilde{\Omega}_\epsilon$ and of $\widetilde{f} = (\widetilde{f}_t)_{t\in[\epsilon,1-\epsilon]}$ so that $\widetilde{f}_\epsilon\in \mathcal{U}$.
    If we take $\mathcal{U}$ sufficiently small, $\widetilde{f}_\epsilon$ and $f_0$ will be isotopic.

    Recall that, for $\epsilon>0$ sufficiently small, $\Sigma^+_{1-\epsilon}$ is the normal graph over $\Sigma^-_t = \Sigma^-_0$ of a smooth function $u_{1-\epsilon}$.
    In other words, we have the equality
    \[ \Sigma^+_{1-\epsilon} = \left\{\exp_x\left(u_{1-\epsilon}(x)N_x\right): x\in \Sigma^-\right\}. \]
    We also point out that $u_{1-\epsilon}>0$ on $\Sigma^-\setminus S$, and vanishes on $S$.

    We now construct an explicit homeomorphism $\Xi:\Omega_{1-\epsilon}\to \bar{B}^{n+1}$ defined by
    \begin{equation}
        \Xi(x,s) = \left( \varphi_1(x),\frac{2s\sqrt{1-|\varphi_1(x)|^2}}{u_{1-\epsilon}(x)} -\sqrt{1-|\varphi_1(x)|^2} \right).
        \label{eq: the homeomorphism Xi}
    \end{equation}
    This means that we first identify $\Omega_{1-\epsilon}$ with the set $\{(x,s)\in\Sigma^-\times\R:0\leq s\leq u_{1-\epsilon}(x)\}$ through the exponential map and then compute $\Xi$ as above.
    The right hand side of \eqref{eq: the homeomorphism Xi} is only defined when $u_{1-\epsilon}(x)>0$, which is true on $\Sigma^-\setminus S$.
    For $(x,0)\in S$, we let $\Xi(x,0) = (\varphi_1(x),0)$.
    
    Note that $\Xi$ is smooth away from the corner $S$ of $\Omega_{1-\epsilon}$.
    In fact, it is a diffeomorphism from $\Omega_{1-\epsilon}\setminus S$ to $\bar{B}^{n+1}\setminus (\Sph^{n-1}\times\{0\})$.
    Observe also that $\Xi(x,u_{1-\epsilon}(x)) = (j\circ \Xi)(x,0) = (j\circ \varphi_1)(x) = \psi_1(x)$ for all $x\in \Sigma^-$.
    Taking $\epsilon>0$ small enough, we can arrange that both functions $\varphi_1^{-1},\varphi_{1-\epsilon}^{-1}:\Sph^n \to M$ are sufficiently close in the $C^\infty$ topology, and also that $\psi^{-1},\psi_{1-\epsilon}^{-1}:\Sph^n\to M$ are sufficiently close in the $C^\infty$ topology.
    This can be made to ensure that, by perturbing $\Xi:\Omega_{1-\epsilon}\to \bar{B}^{n+1}$ slightly, we obtain a new continuous embedding $\widetilde{\Xi}:\Omega_{1-\epsilon}\to \R^{n+1}$ that restricts to a diffeomorphism from $\widetilde{\Omega}_{1-\epsilon}$ to $\bar{B}^{n+1}$, mapping $\widetilde{\Sigma}_{1-\epsilon} = \partial\widetilde{\Omega}_{1-\epsilon}$ to $\Sph^n$ as $(\widetilde{f}_{1-\epsilon})^{-1}:\widetilde{\Sigma}_{1-\epsilon}\to \Sph^n$.

    We then use the isotopy extension theorem to obtain an isotopy $\widetilde{F} = (\widetilde{F}_t)_{t\in[\epsilon,1-\epsilon]}:M\times[\epsilon,1-\epsilon]\to M$ so that $\widetilde{F}_\epsilon = id_M$ and $\widetilde{F}_{t}(\widetilde{f}_\epsilon(y)) = \widetilde{f}_t(y)$ for all $(y,t)\in \Sph^n\times[\epsilon,1-\epsilon]$.
    Note that $\widetilde{F}_{1-\epsilon}$ restricts to a diffeomorphism from $\widetilde{\Omega}_\epsilon$ to $\widetilde{\Omega}_{1-\epsilon}$, and the map $\widetilde{\Xi}\circ\widetilde{F}_{1-\epsilon}:\widetilde{\Omega}_\epsilon\to \bar{B}^{n+1}$ is a diffeomorphism equal to $\widetilde{f}_\epsilon^{-1}$ at the boundary $\partial\widetilde{\Omega}_\epsilon$.
    
    Since $\widetilde{f}_\epsilon$ is isotopic to $f_0 = \psi_0$, there is a diffeomorphism $F':\Omega \to \widetilde{\Omega}_\epsilon$ so that the composition $\widetilde{\Xi}\circ \widetilde{F}_{1-\epsilon}\circ F':\Omega \to \bar{B}^{n+1}$ restricts to $\psi_0$ at $\Sigma = \partial \Omega$.
    Therefore, we proved that there is an extension of $\psi_0 = \psi_\sigma:\Sigma\to \Sph^n$ to a diffeomorphism $\Omega \to \bar{B}^{n+1}$.
    
    This finishes the proof of the theorem, because we may replace $\Omega = \overline{U_0}$ by $F(\Omega,1) = \overline{U_1}$, and extend $\psi_\sigma$ for it as well, so to obtain an extension of $\psi_\sigma$ to a diffeomorphism from $M$ to $\Sph^{n+1}$.
\end{proof}

The same result is obviously false for manifolds admitting oriented transitive families of closed embedded hypersurfaces.
In fact, Example 2.5 shows that any topological classification of manifolds with oriented transitive families is impossible.

Theorem \ref{thm: classification of transitive families of two-sided hypersurfaces} cannot be improved.
For example, by the work of Ambrozio and Guajardo in \cite[see Remark 4.22]{AG}, there exist transitive families of $n$-spheres in $\Sph^{n+1}$ that cannot be identified with the set of equators by any diffeomorphism.

\subsection{Transitive families of one-sided hypersurfaces}
Suppose that $M$ is an $(n+1)$-dimensional manifold that contains a closed embedded one-sided hypersurface $S$.
It is known that $M$ has a double cover $\pi:\widetilde{M}\to M$ with a two-sided pre-image $\widetilde{S}:=\pi^{-1}(S)$, and the restriction of $\pi$ to $\widetilde{S}$ is a double cover $\widetilde{S}\to S$.

In the particular case where $M$ admits a transitive family $\Zoll$ of closed embedded one-sided hypersurfaces and $\Sigma\in\Zoll$, $M$ has a double cover $\pi:\widetilde{M}\to M$ with a two-sided pre-image $\widetilde{\Sigma} = \pi^{-1}(\Sigma)$.
Since all hypersurfaces in $\Zoll$ are isotopic to $\Sigma$, their pre-images under $\pi$ are two-sided and isotopic to $\widetilde{\Sigma}$.

Therefore, $\{\widetilde{\Sigma}=\pi^{-1}(\Sigma):\Sigma\in\Zoll\}$ is a transitive family of closed two-sided embedded hypersurfaces in $\widetilde{M}$.
By Theorem \ref{thm: classification of transitive families of two-sided hypersurfaces}, there exists a diffeomorphism from $\widetilde{M}$ to $\Sph^{n+1}$ that identifies each $\widetilde{\Sigma}$ with an embedded $n$-sphere.

Write $\Zoll = \{\Sigma_\sigma:\sigma\in P\}$.
Since the map $\pi:\widetilde{M}\to M$ induces an injective correspondence $\Sigma\mapsto\widetilde{\Sigma} = \pi^{-1}(\Sigma)$, $P$ also parametrizes the transitive family in $\widetilde{M}$.
In other words, $\{\widetilde{\Sigma}=\pi^{-1}(\Sigma):\Sigma\in\Zoll\} = \{\widetilde{\Sigma}_\sigma = \pi^{-1}(\Sigma_\sigma):\sigma\in P\}$.

As before, we denote by $\mu:G_n(M)\to M$ and $\nu:G_n(M)\to P$ the canonical projections, and by $\widetilde{\mu}$ and $\widetilde{\nu}$ for the analogous maps when we replace $M$ by $\widetilde{M}$.
Recall that we associated with any $\zeta\in T_\sigma P$ a section $\overline{\zeta}\in \Gamma(N\Sigma_\sigma)$.
Our construction did not depend on whether $\Sigma_\sigma$ was two-sided or not, so it is still well defined when $\Sigma_\sigma$ is one-sided.

As in the two-sided case, we can also consider the nodal set $P_{[\zeta]}:= P_{[\zeta]}^\sigma:=\{x\in\Sigma_\sigma:\overline{\zeta}_x = 0\}$.
The difference is that, since the normal bundle $N\Sigma_\sigma$ is nontrivial, $P_{[\zeta]}$ does not separate $\Sigma_\sigma$ into two components.

For any fixed pair $(x,\sigma)\in M\times P$ with $x\in\Sigma_\sigma$ and any Riemannian metric $g$ on $M$, we define the homomorphism
\[ \zeta\in T_\sigma P\mapsto (\overline{\zeta}(x), \nabla\overline{\zeta}(x))\in N_x\Sigma_\sigma\times \left(T^*_x\Sigma_\sigma\otimes N_x\Sigma_\sigma\right).  \]
Passing to the double cover, we can apply Proposition \ref{prop: Zoll dimension reduction} to conclude that this map is an isomorphism.

For each $x\in \Sigma_\sigma$, the evaluation map $\textnormal{ev}_x:\zeta\in T_\sigma P\mapsto \overline{\zeta}_x\in N_x\Sigma_\sigma$ is a surjective linear homomorphism, and its kernel $\ker (\textnormal{ev}_x)$, a well defined element of $\Proj T^*_\sigma P = G_n(T_\sigma P)$.
As we vary $x\in \Sigma_\sigma$, we obtain a smooth map
\[ \varphi_\sigma:x\in\Sigma_\sigma\mapsto \ker(\textnormal{ev}_x)\in\Proj T^*_\sigma P. \]
This is an abstract base-free map analogous to the  $\psi_\sigma$ constructed in the two-sided case.
In fact, passing to the two-sided cover $\widetilde{\Sigma}_\sigma$, Proposition \ref{prop: Sigma is diffeomorphic to S^n through psi_sigma} implies that $\varphi_\sigma:\Sigma_\sigma\to \Proj T^*_\sigma P$ is a diffeomorphism that identifies the transitive family $\{P^\sigma_\xi:\xi\in\Proj T_\sigma P\}$ with the family of projective hyperplanes.

When we take a Riemannian metric $g$ on $M$ and a discontinuous unit normal vector field $N$ along $\Sigma_\sigma$, we may write $\overline{\zeta} = \phi_\zeta \cdot N$ for some discontinuous function $\phi_\zeta:\Sigma_\sigma\to \R$.
If we choose a basis $\zeta_0,...,\zeta_n\in T_\sigma P$, and write $\phi_i = \phi_{\zeta_i}$, we obtain a concrete construction of $\varphi_\sigma$ as $\varphi_\sigma(x) = [\phi_0(x):\cdots:\phi_n(x)]\in \R\Proj^n$.
Note that $[\phi_0:\cdots:\phi_n]$ does not depend on the choice of unit normal $N$, so it is a smooth map.

Passing to the two-sided cover, we consider $\widetilde{g} = \pi^*g$ and fix a choice of unit normal vector field along $\widetilde{\Sigma}_\sigma$ to get a diffeomorphism $\psi_\sigma:\widetilde{\Sigma}_\sigma\to \Sph^n$ induced by the choice of basis $\zeta_0,...,\zeta_n\in T_\sigma P$.
The construction commutes in the sense that the diagram
\begin{equation}
    \begin{matrix}
        \widetilde{\Sigma}_\sigma & \xrightarrow{\psi_\sigma} & \Sph^n \\
        \downarrow& & \downarrow \\
        \Sigma_\sigma & \xrightarrow{\varphi_\sigma} & \R\Proj^n
    \end{matrix}
    \label{diag: psi_sigma and varphi_sigma commute the projections}
\end{equation}
commutes.
Here, $\Sph^n\to \R\Proj^n$ denotes the canonical projection and $\widetilde{\Sigma}_\sigma\to \Sigma$, the map $\pi$.

\begin{theorem}
    An $(n+1)$-manifold equipped with an unoriented transitive family $\Zoll$ of one-sided closed embedded hypersurfaces is diffeomorphic to $\R\Proj^{n+1}$ in such a way that the elements of $\Zoll$ are isotopic to the hyperplanes.
    \label{thm: classification of transitive families of one-sided hypersurfaces}
\end{theorem}
\begin{proof}
    Fix $\Sigma_\sigma\in \Zoll$ and let $\widetilde{\Sigma}_\sigma$ be its two-sided cover.
    By Theorem \ref{thm: classification of transitive families of two-sided hypersurfaces}, $\widetilde{\Sigma}_\sigma$ is diffeomorphic to an $n$-sphere and separates $\widetilde{M}$ in two connected components.
    As in the proof of Part I of Theorem \ref{thm: classification of transitive families of two-sided hypersurfaces}, we denote by $\Omega$ the closure of one of the components of $\widetilde{M}\setminus\widetilde{\Sigma}_\sigma$.

    Since $\Sigma_\sigma$ is one-sided and $\pi:\widetilde{\Sigma}_\sigma\to \Sigma_\sigma$ is a two-sided cover, we see that $M = \pi(\Omega)$.
    In fact, $\pi$ is a diffeomorphism from the interior of $\Omega$ to its image, and $M$ is obtained (up to diffeomorphism) by identifying two points in $\partial\Omega = \widetilde{\Sigma}_\sigma$ that have the same image under $\pi$.

    By Part III of Theorem \ref{thm: classification of transitive families of two-sided hypersurfaces}, $\psi_\sigma$ extends to a diffeomorphism from $\Omega$ to the closed $(n+1)$-ball $\bar{B}^{n+1}$.
    At the boundary, the covering map $\pi:\widetilde{\Sigma}_\sigma\to \Sigma_\sigma$ is transformed under $\psi_\sigma$ and $\varphi_\sigma$ into the canonical projection $\Sph^n\to \R\Proj^n$, as expressed by the commutative diagram \eqref{diag: psi_sigma and varphi_sigma commute the projections}.
    Thus, $M$ is diffeomorphic to the manifold obtained by taking the closed $(n+1)$-ball and identifying two antipodal points on its boundary.
    This is exactly $\R\Proj^{n+1}$, and the hyperplane at infinity in this construction is $\Sigma_\sigma$.
\end{proof}


\section{Duality induced by transitive families}

Let $M$ be a smooth connected $(n+1)$-dimensional manifold equipped with an unoriented transitive family $\Zoll = \{\Sigma_\sigma:\sigma\in P\}$ of closed embedded hypersurfaces.
In the previous section, we successfully classified $M$ up to diffeomorphism.
Now we study the parameter space $P$.
The methods here are inspired by \cite{LM}, but we always assume that $n+1\geq 3$.

For each $\sigma\in P$, we constructed a covering map $\varphi_\sigma:\Sigma_\sigma\to \Proj T^*_\sigma P$.
This has the following property: two points $x,y\in\Sigma_\sigma$ have the same image under $\varphi_\sigma$ if and only if all infinitesimal deformations of $\Sigma_\sigma$ along $\Zoll$ that vanish at $x$ also vanish at $y$.
We shall call two such points $x,y \in \Sigma_\sigma$ \textit{strongly conjugate along $\Sigma_\sigma$ for variations in $\Zoll$}.

The possible topologies for $M$ give us a dichotomy.
In the case $M \approx \Sph^{n+1}$, each $\Sigma_\sigma$ is an $n$-sphere and, for every $\sigma\in P$ and $x\in \Sigma_\sigma$, the set of all $y\in \Sigma_\sigma$ strongly conjugate to $x$ along $\Sigma_\sigma$ for variations in $\Zoll$ is $\varphi_\sigma^{-1}(\varphi_\sigma(x))$ and consists of exatly two points.
Otherwise, each $\Sigma_\sigma$ is a projective $n$-space and, for every $\sigma\in P$ and $x\in \Sigma_\sigma$ there is no point aside from $x$ that is strongly conjugate to $x$ along $\Sigma_\sigma$.

There is a nice way to see that $\varphi_\sigma$ depends smoothly on $\sigma\in P$.
In fact, we can view each $\varphi_\sigma$ as a restriction of a more general smooth map $\varphi:G_n(M)\to G_n(P)$ to the Legendrian lift of $\Sigma_\sigma$.

The map $\varphi:G_n(M)\to G_n(P)$ in question takes a point $z\in G_n(M)$ and sends to $\varphi(z):=\nu_{*,z}(\ker\mu_{*,z})$, where $\mu:G_n(M)\to M$ and $\nu:G_n(M)\to P$ are the canonical projections.
It is well defined because $\ker\mu_*\cap \ker\nu_* = 0$, and the composition of $\varphi$ with the canonical projection $\pi:G_n(P)\to P$ is equal to $\nu$.

By definition of $\varphi$, if $z\in G_n(M)$ and $\nu(z) = \sigma$, then $\varphi(z)\in \Proj T^*_{\sigma}P = G_n(T_{\sigma}P)$ is the set of all $\zeta\in T_{\sigma}P$ whose associated variation $\overline{\zeta}\in\Gamma(N\Sigma_\sigma)$ vanishes at $\mu(z)$.
This identifies $\varphi_\sigma$ as the composition of the Legendrian lift $x\in \Sigma_\sigma\mapsto (x,T_x\Sigma_\sigma)\in G_n(M)$ with $\varphi$.

\begin{proposition}
    $\varphi:G_n(M)\to G_n(P)$ is a covering map of order equal to the order of the covering $\varphi_\sigma:\Sigma_\sigma\to G_n(T_\sigma P)$ for every $\sigma\in P$.
    \label{prop: varphi is a covering map}
\end{proposition}
\begin{proof}
    By Theorem \ref{thm: classification of transitive families of two-sided hypersurfaces} and Theorem \ref{thm: classification of transitive families of one-sided hypersurfaces}, $G_n(M)$ and $P$ are connected closed manifolds.
    Consequently, $G_n(P)$ is also connected and has the same dimension as $G_n(M)$.

    Thus, we need only to prove that $\varphi$ has no critical points.
    For this, we first note that, since $\pi\circ\varphi = \nu$ and $\nu:G_n(M)\to P$ is a fiber bundle, $\ker\varphi_*\subset\ker\nu_*$.
    On the other hand, $\ker\nu_*$ is the tangent space of the foliation $\{\mathcal{L}_{\Sigma_\sigma}:\sigma\in P\}$, generated by Legendrian lifts of the surfaces in $\Zoll$, and the composition $x\in\Sigma_\sigma\mapsto(x,T_x\Sigma_\sigma)\in G_n(M)\mapsto\varphi(x,T_x\Sigma_\sigma)\in G_n(P)$ is the map $\varphi_\sigma$, as argued above.
    Since $\varphi_\sigma$ is a covering map, $\ker\varphi_*\cap\ker\nu_* = 0$, and this proves that $\ker\varphi_* = 0$, which implies that $\varphi$ is a covering map.

    For the last assertion, we note that the equality $\pi\circ\varphi = \nu$ also implies that, if $z,w\in G_n(M)$ and $\varphi(z) = \varphi(w)$, then $\nu(z) = \nu(w)$.
    Therefore, such points $z$ and $w$ must be in the Legendrian lift of $\Sigma_\sigma$, where $\sigma = \nu(z) = \nu(w)$.
    Along $\mathcal{L}_{\Sigma_\sigma}$, $\varphi$ is described as $\varphi_\sigma$.
    Thus, the order of the covering $\varphi$ is the order of the covering $\varphi_\sigma$.
\end{proof}

When $M\approx\R\Proj^{n+1}$, we get a strong form of duality between $M$ and $P$ that closely resembles the duality between points and hyperplanes in projective geometry.
We formalize this in the following two corollaries.

For the next result, recall that the tautological bundle $E\to G_n(P)$ is the vector bundle of all triples $(\sigma,\Pi,\zeta)$ such that $\sigma\in P$, $\Pi\in G_n(T_\sigma P)$ and $\zeta\in \Pi\subset T_\sigma P$.

\begin{corollary}
    If $\Zoll = \{\Sigma_\sigma:\sigma\in P\}$ is an unoriented transitive family of closed embedded hypersurfaces in $M\approx\R\Proj^{n+1}$, then $\varphi:G_n(M)\to G_n(P)$ is a diffeomorphism that induces an isomorphism from the vector bundle $\ker\mu_*\to G_n(M)$ to the tautological bundle $E\to G_n(P)$.
    \label{cor: varphi is a diffeomorphism when M=RP^n+1}
\end{corollary}
\begin{proof}
    By Proposition \ref{prop: varphi is a covering map}, the order of the covering $\varphi$ is the order of the covering $\varphi_\sigma$ for any $\sigma\in P$.
    Since $M\approx\R\Proj^{n+1}$, we know from the previous section that $\varphi_\sigma$ is a diffeomorphism.
    Therefore, the same is true for $\varphi$.

    The second statement follows from the first by noting that the induced map takes $(z,u)\in\ker\mu_*$ to $(\nu_{*,z}(\ker\mu_{*,z}),\nu_{*,z}u)\in E$.
\end{proof}

\begin{corollary}
    Let $M$ be a connected $(n+1)$-dimensional manifold equipped with an unoriented transitive family $\Zoll = \{\Sigma_\sigma:\sigma\in P\}$ of closed embedded hypersurfaces.
    If $M\approx\R\Proj^{n+1}$, then the parameter space $P$ has an unoriented transitive family $\Zoll^*=\{\Sigma^*_x:x\in M\}$ of real projective $n$-spaces, where $\Sigma^*_x = \nu\left(\mu^{-1}(x)\right)$.
    In particular, $P$ is diffeomorphic to $\R\Proj^{n+1}$.
    \label{cor: Theorem B for M=RP^(n+1)}
\end{corollary}
\begin{proof}
    Since $\ker\mu_*\cap \ker\nu_* = 0$, each $\Sigma^*_x$ is an immersed real projective $n$-space in $P$.
    To show that they are embedded, we note that if $x\in M$ and $z,w\in \mu^{-1}(x) = G_n(T_xM)$ are such that $\nu(z) = \nu(w) =:\sigma$, then the Legendrian lift $\mathcal{L}_{\Sigma_\sigma}$ contains both $z$ and $w$.
    Since $\Sigma_\sigma$ is embedded, we conclude that $z = w = (x,T_x\Sigma_\sigma)$.

    By Corollary \ref{cor: varphi is a diffeomorphism when M=RP^n+1}, $\varphi$ is a diffeomorphism between $G_n(M)$ and $G_n(P)$ that identifies the projection $\nu:G_n(M)\to P$ with the canonical projection $\pi:G_n(P)\to P$, in the sense that $\pi\circ\varphi = \nu$.
    In particular, $\varphi\left(\mu^{-1}(x)\right)$ is the Legendrian lift of $\Sigma^*_x$ for every $x\in M$.
    Since $\mu:G_n(M)\to M$ is a fiber bundle, we conclude that $\Zoll^* = \{\Sigma^*_x:x\in M\}$ is an unoriented transitive family of embedded real projective $n$-spaces in $P$.
    By Theorems \ref{thm: classification of transitive families of two-sided hypersurfaces} and \ref{thm: classification of transitive families of one-sided hypersurfaces}, $P$ is diffeomorphic to $\R\Proj^{n+1}$.
\end{proof}

The remaining case of $M\approx \Sph^{n+1}$ turns out to be more subtle.
We still have a family $\Zoll^*=\{\Sigma^*_x:x\in M\}$ of closed embedded real projective $n$-spaces, but it will not be an unoriented transitive family in general.

To see where the difficulty lies, observe that $\varphi:G_n(M)\to G_n(P)$ is a double covering map.
Therefore, if two distinct points $z,w\in G_n(M)$ have the same image $\varphi(z) = \varphi(w)$, their projections $x = \mu(z)$ and $y = \mu(w)$ induce dual hyperplanes $\Sigma^*_x$ and $\Sigma^*_y$ that pass through $\sigma = \nu(z) = \nu(w)$ with the same tangent space $T_\sigma \Sigma^*_x = T_\sigma \Sigma^*_y$.
The problem is that there is no reason to expect that the equality $\Sigma^*_x= \Sigma^*_y$ holds.

The correct way to think about this is as in the duality between points and oriented equators in $\Sph^{n+1}$.
However, the cost is a weaker topological rigidity.

Suppose $M\approx \Sph^{n+1}$, and fix an orientation on $M$.
In this case, each $\Sigma\in\Zoll$ is an embedded $n$-sphere.
We define a double cover $P^+$ of $P$ as the set of pairs $(\sigma,\mathcal{O})$, where $\mathcal{O}$ is a choice of orientation for $\Sigma_\sigma$, and note that there is an $n$-sphere bundle $\nu^+:G^+_n(M)\to P^+$ that collapses the oriented Legendrian lift of $(\Sigma_\sigma,\mathcal{O})$ to the point $(\sigma,\mathcal{O})$.

The canonical projection $\mu^+:G_n(M)\to M$ is an $n$-sphere bundle over $M$.
Therefore, it induces a long exact sequence in homotopy groups
\[ \cdots\to\pi_k(\Sph^n)\to\pi_k(G^+_n(M))\to\pi_k(M)\to\pi_{k-1}(\Sph^n)\to\cdots, \]
from which we conclude that $\pi_k(G^+_n(M)) = 0$ for $1\leq k\leq n-1$, because $M\approx \Sph^{n+1}$.
Replacing this in the long exact sequence
\[ \cdots\to\pi_k(\Sph^n)\to\pi_k(G^+_n(M))\to\pi_k(P^+)\to\pi_{k-1}(\Sph^n)\to\cdots \]
induced by $\nu^+:G^+_n(M)\to P^+$, we see that $\pi_k(P^+) = 0$ for $1\leq k\leq n-1$.

Thus, it follows from the solution to the Generalized Poincar\'e Conjecture that $P^+$ is homeomorphic to the $(n+1)$-sphere.
In particular, this implies that $P$ has the same homotopy groups and the same cohomology ring with $\Z_2$ coefficients as $\R\Proj^{n+1}$.

Fix an orientation for $G^+_n(M)$.
Since $M$ is oriented, this is equivalent to giving smoothly compatible orientations on the fibers of $G^+_n(M)$.
Therefore, we regard each $(\mu^+)^{-1}(x)$ as an oriented embedded $n$-sphere in $G^+_n(M)$ and $\ker\mu^+_*$ as an oriented sub-bundle of $TG^+_n(M)$.

For each $x\in M$, we consider the oriented immersed $n$-sphere $\Sigma^+_x:= \nu^+\left( (\mu^+)^{-1}(x) \right)\subset P^+$.
It is immersed, because $\ker\mu^+_*\cap\ker\nu^+_* = 0$.
In fact, it is embedded, because two points $z,w\in (\mu^+)^{-1}(x)$ that have the same image $\nu^+(z) = \nu^+(w) = (\sigma,\mathcal{O})$ lie in the oriented Legendrian lift of $(\Sigma_\sigma,\mathcal{O})$.
Since $\Sigma_\sigma$ is embedded, $z=w = (x,T_x\Sigma_\sigma,\mathcal{O}_x)$.

\begin{proposition}
    The map $\psi:z\in G^+_n(M)\mapsto \nu^+_{*,z}(\ker\mu^+_{*,z})\in G^+_n(P^+)$ is a diffeomorphism.
    Consequently, the collection $\Zoll^+:=\{\Sigma^+_x:x\in M\}$ is an oriented transitive family of embedded spheres in $P^+$.
    \label{prop: psi is a covering map}
\end{proposition}
\begin{proof}
    We note that the map is well-defined, because $\ker\mu^+_{*,z}$ is oriented for each $z\in G^+_n(M)$.
    It is a covering map, because $\varphi:G_n(M)\to G_n(P)$ is a covering map, and because the diagram
    \[
        \begin{matrix}
            G^+_n(M) & \xrightarrow{\psi} & G^+_n(P^+) & \rightarrow & P^+ \\
            \downarrow & & \downarrow & & \downarrow \\
            G_n(M) & \xrightarrow{\varphi} & G_n(P) & \rightarrow & P
        \end{matrix}
    \]
    commutes and all its down-arrows are coverings.
    Note that the order of $\varphi$ is two, the map $G^+_n(M)\to G_n(M)$ is a double cover and $G^+_n(P^+)\to G_n(P)$ is four-to-one.
    Together, these force $\psi$ to be a diffeomorphism.

    The last statement now follows because the fibers of $G^+_n(M)$ form a foliation and are mapped, under $\psi$, to the oriented Legendrian lifts of the elements of $\Zoll^+$.
\end{proof}

\begin{proposition}
    Let $M$ be a smooth connected $(n+1)$-dimensional manifold equipped with an unoriented transitive family $\Zoll = \{\Sigma_\sigma:\sigma\in P\}$ of closed embedded hypersurfaces.
    If $M\approx \Sph^{n+1}$, then $P$ is homeomorphic to $\R\Proj^{n+1}$.
    \label{prop: P is homeomorphic to RP^(n+1) when M=S^(n+1)}
\end{proposition}
\begin{proof}
    We already know that $P$ has a double cover $P^+$ homeomorphic to $\Sph^{n+1}$.
    Moreover, for any fixed $x\in M$, $\Sigma^*_x = \nu\left(\mu^{-1}(x)\right)$ and $\Sigma^+_x$ are, respectively, an embedded real projective $n$-space in $P$ and an embedded $n$-sphere in $P^+$.
    Note that $\Sigma^+_x$ is the pre-image of $\Sigma^*_x$ under the canonical projection $P^+\to P$.

    By the Schoenflies Theorem (see \cite[Chapter IV, Theorem 19.11]{Bre}), $\Sigma^+_x$ divides $P^+$ into two components, each homeomorphic to the open $(n+1)$-ball.
    Let $\Omega$ be the closure of one of them.
    Then $\Omega$ has the structure of a smooth manifold with boundary $\partial\Omega = \Sigma^+_x$ that is homeomorphic to the closed $(n+1)$-ball.
    Since the covering $P^+\to P$ restricts to a double cover $\Sigma^+_x\to \Sigma^*_x$, and since $\Sigma^*_x$ is one-sided, we see that $P^+\to P$ maps the interior of $\Omega$ diffeomorphically onto its image, and $P$ is obtained (up to diffeomorphism) by identifying the points in the boundary of $\Omega$ having the same image under $\Sigma^+_x\to \Sigma^*_x$.
    Thus, $P$ is homeomorphic to $\R\Proj^{n+1}$.
\end{proof}

\section{Almgren-Pitts theory and surface Zoll manifolds}

We now move from a differential topological to a differential geometric perspective.
This section serves to fix the notation used in the Almgren-Pitts min-max theory, and to give the definition and basic properties of hypersurface Zoll manifolds, especially in the three-dimensional case.

\subsection{Preliminaries on Almgren-Pitts theory}
Given a closed Riemannian $(n+1)$-dimensional manifold $(M,g)$, we denote by $\I_k(M,\Z_2)$ the space of its mod-2 flat $k$-chains with the topology induced by the flat metric $\Flat$, and by $\ZZ_n(M,\Z_2)$ the space of mod-2 flat $n$-boundaries, which is the subspace of those $T = \partial U\in\I_n(M,\Z_2)$ for some $U\in \I_{n+1}(M,\Z_2)$.
We write $\mathcal{V}_n(M)$ for the space of $n$-dimensional rectifiable varifolds in $M$ with the varifold topology induced by the $\FF$-metric.
For any $T\in\I_n(M,\Z_2)$, we let $|T|\in\mathcal{V}_n(M)$ be its induced varifold, and consider the $\FF$-metric in $\ZZ_n(M,\Z_2)$ given by $\FF(S,T) = \Flat(S,T) + \FF(|S|,|T|)$ for $S,T\in\ZZ_n(M,\Z_2)$.
The space $\ZZ_n(M,\FF,\Z_2)$ denotes the set of mod-2 flat $n$-boundaries with the topology induced by the $\FF$-metric.
The mass of a chain $T$ is $\Mass(T)$ and the mass of a varifold $V$ is $||V||(M)$.

The Almgren isomorphism states that $\ZZ_n(M,\Z_2)$ is weakly homotopic equivalent to $\R\Proj^\infty$.
Therefore, 
\[ H^*(\ZZ_n(M,\Z_2);\Z_2) = \Z_2[\overline{\lambda}] \]
where $\overline{\lambda}$ is the nontrivial element of $H^1(\ZZ_n(M,\Z_2);\Z_2)$.

A $k$-sweepout of $M$ is a continuous map $\Phi:X\to \ZZ_n(M,\FF,\Z_2)$, with domain $X = \textnormal{dmn}(\Phi)$ a finite cubical complex, whose pullback 
\[ \Phi^*\left(\overline{\lambda}^k\right)\neq 0 \]
in $H^k(X;\Z_2)$.
We let $\Pcal_k$ be the set of all $k$-sweepouts, and define the $k$-width of $(M,g)$ as
\[ \omega_k(M,g) := \inf_{\Phi\in \Pcal_k} \max_{x\in\textnormal{dmn}(\Phi)}\Mass(\Phi(x)). \]
The sequence $\{\omega_k(M,g)\}_{k=1}^\infty$ is called the volume spectrum of $(M,g)$, and satisfies the Weyl law \cite{LMN}
\[ \lim_{k\to\infty} \omega_k(M,g) k^{-\frac{1}{n+1}} = a(n)vol(M,g)^{\frac{n}{n+1}}. \]
Here, $a(n)$ is a dimensional constant.

An optimal sequence for $\omega_k(M,g)$ is a sequence of sweepouts $\{\Phi_i\}_i\subset \Pcal_k$ such that
\[ \lim_{i\to\infty} \max_{x\in\textnormal{dmn}(\Phi_i)} \Mass(\Phi_i(x)) = \omega_k(M,g). \]
For such a sequence $\{\Phi_i\}$, we define its image set $\Lambda(\{\Phi_i\})$ as the set of varifolds $V\in\mathcal{V}_n(M)$ such that, for some subsequence $\{j\}\subset\{i\}$ and $x_j\in\textnormal{dmn}(\Phi_j)$,
\[ \lim_{j\to\infty}\FF(|\Phi_j(x_j)|,V) = 0. \]
The critical set of $\{\Phi_i\}$ is the subset $C(\{\Phi_i\})$ of those $V\in \Lambda(\{\Phi_i\})$ with mass $||V||(M) = \omega_k(M,g)$.
For $n+1 = 2$, the Almgren-Pitts min-max theory states that there is a stationary geodesic network with integer multiplicities $V\in C(\{\Phi_i\})$.
If $n+1\geq 3$, the Almgren-Pitts min-max theory says that there is some $V\in C(\{\Phi_i\})$ written as the sum of disjoint closed minimal hypersurfaces smoothly embedded outside a set of codimension seven and with integer multiplicities.
We refer to \cite{Pitts} and \cite{MN-Ricci} for these results.

\subsection{Sweepouts by projective planes in $\R\Proj^3$}
Let $\Pcal$ be the space of all smoothly embedded real projective planes in $\R\Proj^3$ with the $C^\infty$ topology.
By \cite{BK1} and \cite{BK2} (see also \cite{KetLio}), $\Pcal$ deformation retracts to $\R\Proj^3$.
Hence,
\[ H^*(\Pcal;\Z_2) = \Z_2[\lambda]/\langle\lambda^4\rangle, \]
where $\lambda\in H^1(\Pcal;\Z_2)$ is the nontrivial element.

A $k$-sweepout of $\R\Proj^3$ by projective planes is a continuous map $\phi:X\to \Pcal$, where $X = \textnormal{dmn}(\phi)$ is a finite cubical complex, such that $\phi^*(\lambda^k)\neq 0$ in $H^k(X;\Z_2)$.
Sweepouts by projective planes only exist for $k\in\{1,2,3\}$.
The set of $k$-sweepouts by projective planes is denoted as $\Pcal'_k$.

Let $g$ be a metric on $\R\Proj^3$.
As in the Almgren-Pitts theory, we define the projective $k$-width of $(\R\Proj^3,g)$ as
\[ \sigma_k(\R\Proj^3,g) := \inf_{\phi\in\Pcal'_k} \max_{x\in \textnormal{dmn}(\phi)}area_g(\phi(x)). \]
We also set the projective $0$-width
\[ \sigma_0(\R\Proj^3,g) := \inf_{\Sigma\in\Pcal} area_g(\Sigma), \]
and note that $\sigma_0(\R\Proj^3,g)>0$ by the results in \cite{BBN}.
As in the Almgren-Pitts theory, a sequence of $k$-sweepouts $\{\phi_i\}\subset\Pcal'_k$ is optimal for $\sigma_k(\R\Proj^3,g)$ if
\[ \sup_{x\in\textnormal{dmn}(\phi_i)}area_g(\phi(x))\to \sigma_k(\R\Proj^3,g) \]
when $k\to\infty$.

We only defined sweepouts in the Almgren-Pitts setting for boundaries.
It is not hard to extend the definition to general modulo two flat cycles, and obtain the definition of sweepouts by projective planes as a special case.
Instead of doing so, we relate the two definitions as follows.
The pre-image of any $\Sigma\in\Pcal$ under the canonical projection $\pi:\Sph^3\to\R\Proj^3$ is an embedded sphere, and we think of it as a mod-2 flat boundary.
With this notation, one can check that a continuous map $\phi:X\to\Pcal$ is a $k$-sweepout by projective planes if and only if the induced map $\Phi:x\in X\mapsto \pi^{-1}(\phi(x))\in \ZZ_2(\Sph^3,\Z_2)$ is a $k$-sweepout.

By the arguments in \cite{BBN}, $\sigma_0(\R\Proj^3,g)$ is the area of some projective plane $\Sigma\in\Pcal$.
This follows from the results in \cite{MSY}.
On the other hand, the Simon-Smith min-max theory (see \cite{ColdLellis}, \cite{LellisPella} and \cite{Ket}) tells that each projective $k$-width is realized as a sum
\[ \sigma_k(\R\Proj^3,g) = \sum_{i=1}^{n_k} m^{(k)}_iarea\left( \Sigma^{(k)}_i\right), \]
where $\{\Sigma^{(k)}_1,...,\Sigma^{(k)}_{n_k}\}$ is a collection of disjoint smoothly embedded minimal surfaces diffeomorphic to the sphere or to the real projective plane, and $m^{(k)}_i$ are integer multiplicities.

We now turn to the relation between projective widths and surface Zoll metrics.
First, we formalize the definition of a hypersurface Zoll manifold already mentioned in the introduction.

\begin{definition}
    A \textit{hypersurface Zoll family} (or a \textit{Zoll family of hypersurfaces}) in a Riemannian manifold $(M,g)$ is an unoriented transitive family $\Zoll$ of closed embedded minimal hypersurfaces.
    We call the triple $(M,g,\Zoll)$ a \textit{hypersurface Zoll manifold} and $g$, a hypersurface Zoll metric.
    When $M$ has dimension three, we call $\Zoll$ a surface Zoll family, $g$ a surface Zoll metric and $(M,g,\Zoll)$ a surface Zoll manifold.
    \label{def: hypersurface Zoll manifolds}
\end{definition}

As a consequence of Theorems A and B, an unoriented transitive family of closed two-sided embedded hypersurfaces $\Zoll=\{\Sigma_\sigma:\sigma\in P\}$ in an $(n+1)$-manifold $M$ gives rise to an $(n+1)$-sweepout by $n$-spheres $\Phi:\sigma\in P\mapsto\Sigma_\sigma\in\ZZ_n(M,\Z_2)$.
Indeed, for a fixed $\Sigma\in\Zoll$, we can rotate $\Sigma$ as in the proof of Theorem \ref{thm: classification of transitive families of two-sided hypersurfaces} and obtain a path $t\in[0,1]\mapsto\Sigma_t\in\Zoll$ such that $\Sigma=\Sigma_0 = \Sigma_1$, $\Sigma_t = \partial \overline{U_t}$ and $\overline{U_1} = M\setminus U_0$.
This shows that $\Phi^*(\overline{\lambda})\neq 0$ in $H^1(P;\Z_2)$.
Therefore, $\Phi^*(\overline{\lambda}^{n+1}) \neq 0$, since $P$ is homeomorphic to $\R\Proj^{n+1}$.

If the hypersurfaces in $\Zoll$ are one-sided, then $M\approx\R\Proj^n$ and $\sigma\mapsto\Sigma_\sigma$ is an $(n+1)$-sweepout of $M$ by projective hyperplanes in the sense that, if we pass to the universal cover $\pi:\widetilde{M}\approx\Sph^n\to M$, $\Phi:\sigma\in P\mapsto\pi^{-1}(\Sigma_\sigma)\in\ZZ_n(\widetilde{M},\Z_2)$ is an $(n+1)$-sweepout by $n$-spheres.

When $\Zoll$ is a surface Zoll family in $(M^3,g)$, Theorems A and B and the G\'alvez-Mira uniqueness theorem allow us to get a refined geometric information.
The case of $M\approx\Sph^3$ was done in \cite{AMN2}, and we state below.

\begin{proposition}[Proposition 5.1 of \cite{AMN2}]
    If $\Zoll=\{\widetilde{\Sigma}_\sigma:\sigma\in P\}$ is a surface Zoll family in $(\Sph^3,\widetilde{g})$, then $\widetilde{\Sigma}_\sigma$ has Morse index one and nullity equal to three for every $\sigma\in P$.
    Moreover, every immersed minimal 2-sphere in $(\Sph^3,g)$ is embedded and belongs to $\Zoll$.
    \label{prop: prop 5.1 of AMN2}
\end{proposition}

We point out how to modify this in the case of $\R\Proj^3$.

\begin{proposition}
    If $\{\Sigma_\sigma:\sigma\in P\}$ is a surface Zoll family in $(\R\Proj^3,g)$, then each $\Sigma_\sigma$ is stable and has nullity three.
    In fact, each $\Sigma_\sigma$ minimizes area over the set $\Pcal$, and every immersed minimal 2-sphere in $(\R\Proj^3,g)$ is a double covers of an element of the surface Zoll family.
    As a consequence, every immersed minimal projective plane $\Sigma\subset \R\Proj^3$ is embedded and belongs to $\{\Sigma_\sigma\}$.
    \label{prop: modified prop 5.1 of AMN2 for RP3}
\end{proposition}
\begin{proof}
    By Theorem A, each $\Sigma_\sigma$ is an embedded projective plane.
    Passing to the Zoll family of 2-spheres $\{\widetilde{\Sigma}_\sigma = \pi^{-1}(\Sigma_\sigma):\sigma\in P\}$ in the Riemannian covering $\pi:(\Sph^3,\widetilde{g})\to(\R\Proj^3,g)$, and using Proposition \ref{prop: prop 5.1 of AMN2}, we see that each $\Sigma_\sigma$ has nullity three.
    Also, by the uniqueness result of \cite{GM}, every immersed $\widetilde{g}$-minimal sphere $\widetilde{\Sigma}\subset \Sph^3$ belongs to the Zoll family.
    When we return to $(\R\Proj^3,g)$, we see that if $i:\Sph^2\looparrowright \R\Proj^3$ is a $g$-minimal immersion, then any lift $i':\Sph^2\to \Sph^3$ of $i$ is a $\widetilde{g}$-minimal embedding and its image is an element of $\{\widetilde{\Sigma}_\sigma:\sigma\in P\}$.
    Therefore, the image of $\Sigma = i(\Sph^2)\subset\R\Proj^3$ is an embedded minimal projective plane that belongs to the family $\{\Sigma_\sigma:\sigma\in P\}$.
    Moreover, $i:\Sph^2\to\Sigma$ is a double cover.

    Lastly, by the results  of \cite{BBN}, $\sigma_0(\R\Proj^3,g)$ is the area of some embedded stable minimal projective plane $\Sigma$ in $(\R\Proj^3,g)$, which belongs to the Zoll family by the last paragraph.
    Since $\sigma\in P\mapsto \Sigma_\sigma\in\Pcal$ is a smooth family of minimal projective planes, the area functional is constant along the Zoll family.
    Thus $area(\Sigma_\sigma) = \sigma_0(\R\Proj^3,g)$ for every $\sigma\in P$.
\end{proof}

\begin{theorem}
    The projective widths of a surface Zoll metric on $\R\Proj^3$ are all equal.
    \label{thr: easy implication}
\end{theorem}
\begin{proof}
    Let $g$ be a surface Zoll metric on $\R\Proj^3$ and write $\{\Sigma_\sigma:\sigma\in P\}$ for its Zoll family.
    By Theorem A, each surface $\Sigma_\sigma$ is a minimal projective plane.
    By Theorem B, by the discussion in the two paragraphs following Definition \ref{def: hypersurface Zoll manifolds}, and by Proposition \ref{prop: modified prop 5.1 of AMN2 for RP3}, the map $\sigma\in P\mapsto \Sigma_\sigma\in\Pcal$ is a 3-sweepout such that $\max_{\sigma\in P}area(\Sigma_\sigma) = \sigma_0(\R\Proj^3,g)$.
    Since 
    \[\sigma_0(\R\Proj^3,g)\leq\sigma_1(\R\Proj^3,g)\leq\sigma_2(\R\Proj^3,g)\leq\sigma_3(\R\Proj^3,g)\leq \max_{\sigma\in P}area(\Sigma_\sigma),\] we conclude that 
    \[ \sigma_0(\R\Proj^3,g) = \sigma_1(\R\Proj^3,g) = \sigma_2(\R\Proj^3,g) = \sigma_3(\R\Proj^3,g). \]
\end{proof}


\section{Metric on $\R\Proj^3$ with equal projective widths}

Recall that Theorem C states that $(\R\Proj^3,g)$ is surface Zoll if and only if it has equal projective widths.
Theorem \ref{thr: easy implication} shows one direction.
In this section, we prove the implication that remains to be proven in order to establish Theorem C.

\subsection{Uniform estimates}
Let $g$ be a Riemannian metric on $\R\Proj^3$ for which
\[ \sigma_0(\R\Proj^3,g) = \sigma_3(\R\Proj^3,g)=:\sigma_0, \]
and take a sequence $\{\phi_i\}_{i=1}^\infty\subset\Pcal'_3$ of $3$-sweepouts that is optimal for $\sigma_0$, in the sense that
\[ \lim_{i\to\infty} \max_{x\in X_i} area(\phi_i(x))=\sigma_0, \]
where $X_i = \text{dmn}(\phi_i)$.
Since our goal is to prove that $g$ is surface Zoll, we need a candidate for the Zoll family in $(\R\Proj^3,g)$ | and Proposition \ref{prop: modified prop 5.1 of AMN2 for RP3} suggests that we consider the set $\mathcal{V}$ of varifolds of the form $|\Sigma|$, where $\Sigma\subset \R\Proj^3$ is an embedded projective plane with $area(\Sigma)=\sigma_0$.

Recall that the results in \cite{BBN} tell us that $\mathcal{V} \neq \emptyset$.
We refine the argument and show that the area of a projective plane $\Sigma\in\Pcal$ gives information about the $\FF$-distance of $|\Sigma|$ to $\mathcal{V}$.

\begin{lemma}
    For every $\epsilon>0$, there is some $\delta>0$ such that
    \[ \FF(|\Sigma|,\mathcal{V})<\epsilon \]
    whenever $\Sigma\in\Pcal$ and $area(\Sigma)<\sigma_0 + \delta$.
    \label{lemma: uniform convergence of phi_i to Vcal in the varifold topology}
\end{lemma}
\begin{proof}
    Suppose that the assertion is not true for the sake of contradiction.
    Then there is some $\epsilon>0$ and a sequence of projective planes $\{\Sigma_i\}\subset \Pcal$ such that $area(\Sigma_i)\to\sigma_0$, but $\FF(|\Sigma_i|,\mathcal{V})\geq\epsilon$ for each $i\geq 1$.
    Since $\{\Sigma_i\}$ is a minimizing sequence for the isotopy class in the sense of \cite{MSY}, we apply the Meeks-Simon-Yau Theorem to get a subsequence $\{\Sigma_{i_j}\}\subset\{\Sigma_i\}$ that converges in varifold to 
    \[ V = m_1|\Gamma_1| + \cdots+m_k|\Gamma_k|, \]
    where $\{\Gamma_1,...,\Gamma_k\}$ is a collection of closed smooth embedded disjoint minimal surfaces.
    Moreover, $\Mass(V) \leq \sigma_0$ by semicontinuity of mass, and each $\Gamma_i$ is either a sphere or a projective plane.

    An important step in the proof that $|\Sigma_{i_j}|\to V$ is a surgery procedure that cuts thin necks of $\Sigma_{i_j}$, where singularities might form in the limit. 
    Whenever this surgery is applied on a nonorientable surface, the result is either a nonrientable surface or two disjoint surfaces, one of them being nonorientable.
    As a consequence, the nonorientability of each $\Sigma_i$ reflects in the limit $V$, for at least one $\Gamma_j$, say $\Gamma_1$, is nonorientable.
    Hence $\Gamma_1$ must be a projective plane with area
    \[ \sigma_0 = \sigma_0(\R\Proj^3,g)\leq area(\Gamma_1)\leq \Mass(V)\leq\sigma_0. \]
    This shows that $k=m_1=1$ and $V = |\Gamma_1|\in\mathcal{V}$.
    But since $|\Sigma_{i_j}|\to V$ as varifolds, the distances $\FF(|\Sigma_{i_j}|,\mathcal{V})\leq \FF(|\Sigma_{i_j}|,V)\to 0$, a contradiction.
\end{proof}

The consequence of Lemma \ref{lemma: uniform convergence of phi_i to Vcal in the varifold topology} is, in a sense, a uniform convergence of $\{\phi_i\}_i$ to $\mathcal{V}$ in the varifold topology.
For this, recall that, as in the Almgren-Pitts setting, the image set of $\{\phi_i\}$ is the set of varifolds $V$ that are the varifold-limit of some sequence $\{\phi_{i_j}(x_j)\}_j$, and we denote it by $\Lambda(\{\phi_i\})$.
The critical set $C(\{\phi_i\})$ is the subset of elements $V\in\Lambda(\{\phi_i\})$ with mass $\Mass(V) = \sigma_0$.

\begin{corollary}
    \[ \lim_{i\to\infty}\max_{x\in X_i} \FF(|\phi_i(x)|,\mathcal{V}) = 0. \]
    In particular, $C(\{\phi_i\})=\Lambda(\{\phi_i\})\subset \mathcal{V}$.
    \label{cor: Critical set = Image set < Vcal}
\end{corollary}

A similar strategy used in the proof of Lemma \ref{lemma: uniform convergence of phi_i to Vcal in the varifold topology} shows the compactness of $\mathcal{V}$ in the varifold topology.
In fact, after fixing an integer $k\geq1$ and some $\alpha\in(0,1)$, Allard's theorem and the compactness of $\mathcal{V}$ imply that the set $\mathcal{Z} = \{\Sigma\in\Pcal:area(\Sigma)=\sigma_0\}$ is compact in the $C^{k,\alpha}$ topology.

It will be convenient to work with the set $\Pcal^{k,\alpha}$ of $C^{k,\alpha}$-embedded projetive planes in $(\R\Proj^3,g)$.
For a given $\Sigma\in\Pcal^{k,\alpha}$ and some $r>0$, we write $B^{k,\alpha}_r(\Sigma)$ for the set of those $\Sigma'\in\Pcal^{k,\alpha}$ that can be written as $\Sigma' = \{\exp_x(X_x): x\in \Sigma\}$ for some normal vector field $X$ along $\Sigma$ with norm $||X||_{k,\alpha}<r$.

We end this subsection with an improvement of the convergence in Corollary \ref{cor: Critical set = Image set < Vcal}.
In this version, we view $\Pcal$ as a subset of flat cycles modulo two, and associate to it both the flat and the $\FF$ topologies, generated respectively by the metrics $\mathcal{F}$ and $\FF$, where $\FF(\Sigma_1,\Sigma_2) = \mathcal{F}(\Sigma_1,\Sigma_2)+\FF(|\Sigma_1|,|\Sigma_2|)$.

\begin{proposition}
    Let $\mathcal{W}$ be a compact subset of $\mathcal{V}$, and denote by $\mathcal{K} = \{\Sigma\in\Pcal:|\Sigma|\in\mathcal{W}\}\subset \mathcal{Z}$.
    If $\Lambda(\{\phi_i\})\subset \mathcal{W}$, then
    \[ \lim_{i\to\infty}\max_{x\in X_i}\FF(\phi_i(x),\mathcal{K}) = 0. \]
    \label{prop: uniform convergence in Flat + F}
\end{proposition}
\begin{proof}
    The same argument in Lemma \ref{lemma: uniform convergence of phi_i to Vcal in the varifold topology} implies that
    \[ \lim_{i\to\infty}\max_{x\in X_i}\FF(|\phi_i(x)|,\mathcal{W}) = 0. \]
    So we only need to prove $\lim_{i\to\infty}\max_{x\in X_i}\mathcal{F}(\phi_i(x),\mathcal{K}) = 0$.
    For this, note that if a sequence $\{\Sigma_i\}\subset \Pcal$ converges as varifolds to some $\Sigma \in \Pcal$ and in the flat metric to some flat cycle $T$, then either $T=0$ or $T=\Sigma$ by the Constancy Theorem.
    
    On the other hand, we know that $\mathcal{K}$ is compact in the $C^{k,\alpha}$ topology by compactness of $\mathcal{W}$ and Allard's Theorem | so it is also $\FF$-compact.
    This, together with the argument in the last paragraph, tell us that, for every $\epsilon>0$, there is a $\delta\in(0,\epsilon)$ such that, if $\Sigma\in \mathcal{K}$ and $\Sigma'\in\mathcal{P}$ satisfy $\FF(|\Sigma|,|\Sigma'|)<\delta$, then either $\mathcal{F}(\Sigma,\Sigma')<\epsilon$ or $\mathcal{F}(\Sigma',0)<\epsilon$.

    Now, if a 3-sweepout $\phi:X\to\Pcal^{k,\alpha}$ is $\epsilon$-close to $\mathcal{W}$ in $\FF$-metric, in the sense that $\FF(|\phi(x)|,\mathcal{W})<\epsilon$ for all $x\in X$, then either $\FF(\phi(X),\mathcal{K})<2\epsilon$ or there is some $x_0\in X$ and some $\Sigma\in\mathcal{K}$ such that $\FF(|\phi(x_0)|,|\Sigma|)<\epsilon$ and $\mathcal{F}(\phi(x_0),0)<\epsilon$.
    But since $\phi$ is $\FF$-continuous (it is continuous with respect to the $C^{k,\alpha}$-topology of $\Pcal^{k,\alpha}$), and since $area(\Sigma')=\sigma_0>0$ for all $\Sigma'\in\mathcal{K}$, we conclude that $\mathcal{F}(\phi(x),0)<\epsilon$ if we take $\epsilon<\sigma_0/2$.

    To see why such a 3-sweepout cannot exist for some $\epsilon$ sufficiently small, note that, if $x:[0,1]\to X$ is a non-contractible loop for which $\phi^*(\lambda)([x])\neq 0$, then $\phi\circ x$ determines a path $t\in [0,1]\mapsto \Omega_t\subset \Sph^3$ of open $C^{k,\alpha}$-domains such that (1) $t\mapsto vol_g(\Omega_t)$ is continuous, (2) $\phi(x(t)) = \pi(\partial\Omega_t)$ for all $t\in[0,1]$ and (3) $\Omega_1 = \Sph^3\setminus\overline{\Omega_0}$.
    For $\epsilon\in(0,vol_g(\Sph^3)/2)$ small, we can further take $\Omega_t$ to satisfy the upper bound $vol_g(\Omega_t)<\epsilon$, and from this we get the two inequalities $vol_g(\Omega_0)+vol_g(\Omega_1) < 2\epsilon<vol_g(\Sph^3)$.
    However, $vol_g(\Sph^3) = vol_g(\Omega_0)+vol_g(\Omega_1)$, and we have a contradiction.

    This completes the proof, because $\max_{x\in X_i}\FF(|\phi_i|,\mathcal{W})<\delta$ for every $i$ sufficiently large.
\end{proof}

\subsection{Existence of a 3-sweepout by minimal projective planes}

We start the second part of the proof with two general results that will give us the final tools for the last implication of Theorem C.

\begin{proposition}
    Suppose $g$ is any metric on $\R\Proj^3$, and let $\pi:(\Sph^3,\widetilde{g})\to (\R\Proj^3,g)$ be the double Riemannian covering.
    If $\Sigma\subset\R\Proj^3$ is an embedded minimal stable projective plane, then $\widetilde{\Sigma}:=\pi^{-1}(\Sigma)\subset \Sph^3$ is an embedded minimal sphere with Morse index at most one and nullity at most three.
    In particular, the nullity of $\Sigma$ is less than or equal to 3.
    \label{prop: nullity upper bound}
\end{proposition}
\begin{proof}
    We will prove that $\widetilde{\Sigma}$ has index at most one.
    By \cite{Cheng}, this will give us the desired bound $nul(\widetilde{\Sigma})\leq 3$.
    Since $nul(\Sigma)\leq nul(\widetilde{\Sigma})$, this also implies that the same upper bound on the nullity of $\Sigma$ holds.
    
    We must prove that the second eigenvalue of the Jacobi operator $L_{\widetilde{\Sigma}}$ is nonnegative.
    Indeed, it is known that the complement of the nodal set of a nontrivial second eigenfunction $u:\widetilde{\Sigma}\to\R$ of $L_{\widetilde{\Sigma}}$ consists of two components, one where $u>0$ and the other where $u<0$.
    By \cite{Cheng}, $\{u=0\}$ is connected and locally the zero set of the polynomial $z\in\C\mapsto Re(z^n)$ for some $n\geq 1$.
    If the nodal set had a singularity, it would divide $\widetilde{\Sigma}$ in more than two components by the Jordan Curve Theorem, which is not the case.
    
    Thus, $\{u=0\}$ is a simple closed curve for every second eigenfunction $u$.
    Moreover, if $u$ is a second eigenfunction, it is uniquely determined by both $u(x)$ and $Du(x)$ for any $x\in\widetilde{\Sigma}$, because the nodal set of an eigenfunction that vanishes to the first order at a point must have a singularity there.
    
    Since the antipodal map $i:x\in \Sph^3\mapsto-x\in\Sph^3$ is an isometry of $(\Sph^3,\widetilde{g})$ that fixes $\widetilde{\Sigma}$ and inverts a unit normal vector along $\widetilde{\Sigma}$, the function $v(x):=-u(-x)$ is also a second eigenfunction.
    Therefore, we proved that the subspace $V\subset W^{1,2}(\widetilde{\Sigma})$ of second eigenfunctions of $L_{\widetilde{\Sigma}}$ has dimension at most three and is equipped with a linear isometry $I:u\in V\mapsto u\circ i \in V$.

    We now claim that there is a nontrivial function $u\in V$ such that $I(u) = u\circ i = \pm u$.
    There are two case to investigate.
    If the dimension of $V$ is odd, this is true because every linear isometry of an odd dimensional vector space has some eigenvalue equal to $\pm1$.
    If the dimension of $V$ is even, then $V$ has dimension two, and every linear isometry of the plane is either a rotation centered at the origin or a reflection across a line of $V$ that passes through the origin.
    If $I$ is a reflection across a line $\ell\subset V$ (\textit{i.e.} $I\in O(2)\setminus SO(2)$), then $I(u) = u$ for every $u\in\ell$.
    If $I$ is a rotation (\textit{i.e.} $I\in SO(2)$), then $I$ is either the identity $id:V\to V$ or $-id$, because $I^2=id$.

    Hence, we can take $u$ to be a second eigenfunction so that $v = - u\circ i = \pm u$.
    For this $u$, the nodal line $\{u = 0\}$ is a simple closed curve invariant for the antipodal map $i$, so $i$ restricts to an orientation preserving isometry of $\{u=0\}$, for it would have some fixed point if it reversed the orientation.
    At the same time, since $i$ preserves the orientation of $\Sph^3$ and fixes $\widetilde{\Sigma}$, but inverts a unit normal vector along $\widetilde{\Sigma}$, it restricts to an orientation reversing isometry of $\widetilde{\Sigma}$.
    Since $T_x\Sph^3 = \{v\in \R^4:x\cdot v = 0\}$, and since $i(\widetilde{\Sigma}) = \widetilde{\Sigma}$, we can identify $T_x\widetilde{\Sigma} = T_{-x}\widetilde{\Sigma}$ for every $x\in\widetilde{\Sigma}$.
    Applying the chain rule, we conclude that $Du(x) = Dv(x)$ at any point $x\in \{u=0\}$, for if $Du(x) = -Dv(x)$ we would conclude that $i$ preserves the orientation of $\widetilde{\Sigma}$, which is not the case.
    This proves that $u=v = -u\circ i$.
    
    Thus, for this choice of $u$, the Jacobi vector field $\widetilde{X}=uN$ descends to an eigensection $X$ along $\Sigma$ of the Jacobi operator $L_\Sigma$.
    By the stability of $\Sigma$, the associated eigenvalue of $X$ is nonnegative.
    Hence, the second eigenvalue of $\widetilde{\Sigma}$ is also nonnegative, and $\widetilde{\Sigma}$ has index less than or equal to one.
\end{proof}

\begin{proposition}
    Let $\Sigma$ be a smooth closed embedded minimal hypersurface in a Riemannian manifold $(M,g)$.
    Fix an integer $k\geq 2$ and a real number $\alpha\in(0,1)$, and write $\mathfrak{X}^\perp_{k,\alpha}(\Sigma)$ for the set of $C^{k,\alpha}$ normal vector fields along $\Sigma$.
    If the nullity of $\Sigma$ is less than or equal to some integer $n\geq 0$, then there exists a smooth embedding $\varphi:\bar{B}^n\to\mathfrak{X}^\perp_{k,\alpha}(\Sigma)$ with $\varphi(0)=0$ and $D\varphi(0) = Id$, and such that every closed minimal hypersurface $\Sigma'\subset M$ sufficiently close to $\Sigma$ in the varifold topology is the graph over $\Sigma$ of some vector field $X=\varphi(w)$ for some $w\in \bar{B}^n$.
    This means that $\Sigma' = \{ \exp_x(X(x)): x \in \Sigma \}$.
    \label{prop: White's deformation}
\end{proposition}
\begin{proof}
    It follows from Theorem 1.3 of \cite{Wht1} that there is a neighborhood $\mathcal{W}\subset \ker L_\Sigma$ of the origin and a smooth embedding $\psi:\mathcal{W}\to \mathfrak{X}^\perp_{k,\alpha}(\Sigma)$ with $\psi(0) = 0$ and $D\psi(0) = Id$, and such that every closed minimal hypersurface $\Sigma'\subset M$ sufficiently close to $\Sigma$ in the varifold topology is the graph over $\Sigma$ of some vector field $X=\psi(w)$ for some $w\in \mathcal{W}$.
    This is true, since Allard regularity implies that any closed minimal hypersurface $\Sigma'$ that is sufficiently close in the varifold topology to $\Sigma$ is close to $\Sigma$ in the smooth topology.
    Our hypothesis implies that $\dim\ker L_\Sigma\leq n$, so we may embed $\ker L_\Sigma\subset \R^n$ linearly and take $\epsilon>0$ so small that $\bar{B}^n_\epsilon\cap \ker L_\Sigma = \bar{B}^n_\epsilon\cap\mathcal{W}$.
    It is now easy to extend $\psi$ to obtain the embedding $\varphi$ as required, since $\mathfrak{X}^\perp_{k,\alpha}(\Sigma)$ is an infinite dimensional Banach space. 
\end{proof}

With the notation of subsection 5.1, let $g$ be a metric with equal projective widths, and denote by 
\[\mathcal{Z} = \{\Sigma\in\Pcal:area(\Sigma)=\sigma_0\}.\]
We also write $\mathcal{V}$ for the set of varifolds $|\Sigma|$ for $\Sigma\in\mathcal{Z}$.

It follows from Propositions \ref{prop: nullity upper bound} and \ref{prop: White's deformation} that, associated to every $\Sigma\in\mathcal{Z}$, there is a $\delta_\Sigma>0$ and an embedding $\varphi_\Sigma$, from the closed unit ball $\bar{B}^3$ into $\Pcal^{k,\alpha}$, such that $\varphi_\Sigma(0) = \Sigma$ and $\Sigma'\in\mathcal{B}_\Sigma = \varphi_\Sigma(\bar{B}^3)$ for every closed minimal surface $\Sigma'$ with $\FF(|\Sigma'|,|\Sigma|)<\delta_\Sigma$.
We define the subset $\mathcal{G}\subset\mathcal{Z}$ as the set of those $\Sigma\in\mathcal{Z}$ for which $\varphi_\Sigma$ and $\delta_\Sigma$ can be chosen so that $\varphi_\Sigma(z)$ is a minimal surface for every $z\in\bar{B}^3$.
Note that $\mathcal{G}$ has the structure of a 3-manifold and is an open subset of $\mathcal{Z}$.

As in Proposition \ref{prop: uniform convergence in Flat + F}, we assume that $\{\phi_i\}\subset\Pcal'_3$ is an optimal sequence of 3-sweepouts whose image $\Lambda(\{\phi_i\})$ lies in some compact subset $\mathcal{W}\subset\mathcal{V}$.
We also write $\mathcal{K}=\{\Sigma\in\mathcal{Z}:|\Sigma|\in\mathcal{W}\}$.

The results of Propositions \ref{prop: uniform convergence in Flat + F}, \ref{prop: nullity upper bound} and \ref{prop: White's deformation} show that we are in position to apply all the arguments in \cite[Section 6.2]{AMN2}.
In particular, we can follow the proof of \cite[Propositions 6.7, 6.8, 6.9 and 6.10]{AMN2} to obtain the following result.

\begin{proposition}
    There is a sequence of 3-sweepouts $\psi_i:Y_i\to \mathcal{G}$ and a compact set $\mathcal{K}_0\subset\mathcal{G}$ such that
    \[\lim_{i\to\infty}\sup_{y\in Y_i}\FF(\psi_i(y),\mathcal{K}_0) = 0. \]
    Moreover each $Y_i$ is a compact connected 3-dimensional simplicial complex.
    \label{prop: sweepouts with image in G}
\end{proposition}
\begin{proof}
    This follows from the same arguments used in the proof of \cite[Propositions 6.7-6.10]{AMN2}.
    The statement above is analogous to the statement of \cite[Proposition 6.10]{AMN2}, and the idea is to deform the sequence $\{\phi_i\}$ to obtain the new sequence $\{\psi_i\}$, after some possible changes in the domains of $\phi_i$.

    We also remark that there is a small change that we must do in our case.
    The proof of \cite[Proposition 6.7]{AMN2} starts with a claim that uses the fact that embedded 2-spheres in $\Sph^3$ are two-sided.
    In our case, we work with one-sided embedded surfaces, so the analogous statement is the following.
    \begin{namedtheorem}[Claim]
        Let $n$ be a positive integer and let $\psi:\partial B^n\to\Pcal^{k,\alpha}$ be a continuous map so that, for some $\Sigma\in\mathcal{K}$ and some $r>0$, $\psi(x)\in B^{k,\alpha}_r(\Sigma)$ for every $x\in\partial B^n$.
        Then $\psi$ extends continuously to a map $\bar{B}^n\to B^{k,\alpha}_r(\Sigma)$.
    \end{namedtheorem}
    \begin{proof}
        A map $\psi:\partial B^n\to B^{k,\alpha}_r(\Sigma)$ is equivalent to a continuous function $X:\partial B^n\to \mathfrak{X}^\perp_{k,\alpha}(\Sigma)$ such that $||X(x)||_{k,\alpha}<r$.
        The correspondence is given by writing $\psi(x)$ as the graph of $X(x)$ in the following sense
        \[ \psi(x) = \{ \exp_p(X(x)_p):p\in\Sigma \}. \]
    
        Now, we can take the extension $\overline{\psi}:\bar{B}^n\to B^{k,\alpha}_r(\Sigma)$ as
        \[ \overline{\psi}(tx) = \{ \exp_p(tX(x)_p): p\in\Sigma \} \]
        for $t\in[0,1]$ and $x\in \partial B^n$.
    \end{proof}
    The rest of the proof of \cite[Proposition 6.7]{AMN2} now works in our case to give the analogous result.
\end{proof}

\subsection{End of the proof}

Let $g$ be a Riemannian metric in $\R\Proj^3$ with all projective widths equal to $\sigma_0$, and consider the sequence of sweepouts $\psi_i:Y_i\to\Pcal^{k,\alpha}$ given by Proposition \ref{prop: sweepouts with image in G}.
Since $(\psi_i)^*\lambda^3\neq 0$, there exists some $\alpha\in H_3(Y_i;\Z_2)$ for which $(\psi_i)^*(\lambda^3)\cdot\alpha=1$.
In particular $\lambda^3\cdot(\psi_i)_*\alpha = 1$, and this implies that $\lambda^3_{|\mathcal{G}}\neq 0$.
Hence, $\mathcal{G}$ is a closed 3-dimensional manifold and each $\psi_i$, having mod 2 degree 1, maps $Y_i$ onto a component of $\mathcal{G}$.
Indeed, if $\mathcal{G}$ were open, we would have the vanishing $H^3(\mathcal{G};\Z_2)=0$.

We may pass to a subsequence of $\{\psi_i\}$ and suppose that $\psi_i(Y_i) = \psi_j(Y_j)$ for each $i,j$, and denote this component by $\Zoll$.
Note that $\Zoll$ is a compact subset of $\Pcal$ in the smooth topology and a 3-dimensional closed manifold, so it is a smooth 3-parameter family of minimal projective planes in $(\R\Proj^3,g)$.
Hence, we only need to prove that through any pair $(x,V)\in G_2(\R\Proj^3)$ passes a unique $\Sigma\in \Zoll$, in the sense that $T_x\Sigma = V$.

To see this, consider the set $\Pcal \Zoll$ of triples $(\Sigma,x,\mathcal{O})$ where $\Sigma\in M$, $x\in\Sigma$ and $\mathcal{O}$ is an orientation for $T_x\Sigma$.
This is a compact smooth manifold of dimension 5, and there is a canonical map
\[ T:(\Sigma,x,\mathcal{O})\in\Pcal \Zoll\mapsto (x,T_x\Sigma,\mathcal{O})\in G^+_2(\R\Proj^3), \]
from $\Pcal \Zoll$ to the Grassimanian bundle of oriented 2-planes of $\R\Proj^3$.

We claim that $T$ is locally injective, in the sense that every point $(\Sigma,x,\mathcal{O})$ has a neighborhood $U\subset \Pcal \Zoll$ where $T$ is injective when restricted to $U$.
Indeed, if this was not the case, there would be some $(\Sigma,x,\mathcal{O})$ and two sequences $(\Sigma^1_i,x^1_i,\mathcal{O}^1_i)$ and $(\Sigma^2_i,x^2_i,\mathcal{O}^2_i)$ converging to $(\Sigma,x,\mathcal{O})$ such that each $(\Sigma^1_i,x^1_i,\mathcal{O}^1_i)\neq (\Sigma^2_i,x^2_i,\mathcal{O}^2_i)$, but $T(\Sigma^1_i,x^1_i,\mathcal{O}^1_i) = T(\Sigma^2_i,x^2_i,\mathcal{O}^2_i)$.
Hence, $x^1_i=x^2_i$ and $T_{x^1_i}\Sigma^1_i = T_{x^2_i}\Sigma^2_i$ (as oriented planes) for each $i$.

We lift to the double covering $\pi:(\Sph^3,\widetilde{g})\to(\R\Proj^3,g)$.
When $i$ is sufficiently large, we may write $\widetilde{\Sigma}^2_i = \pi^{-1}(\Sigma^2_i)$ as a graph over $\widetilde{\Sigma}^1_i=\pi^{-1}(\Sigma^1_i)$.
Using elliptic theory applied to the minimal surface equation, we get, in the limit, a nontrivial solution $u\in C^\infty(\widetilde{\Sigma})$ to the Jacobi equation $L_{\widetilde{\Sigma}}u=0$ such that both $u(x) = 0$ and $Du(x) = 0$.
However, we know from the proof of Proposition \ref{prop: nullity upper bound} that the index of $\widetilde{\Sigma}$ is at most 1, so that $u$ is either a first or a second eigenfunction.
Both cases cannot happen, for if $u$ was a nontrivial first eigenfunction, then $u$ would not vanish anywhere; and if $u$ was a second eigenfunction, it would not vanish to the first order at any point.
This contradiction proves that $T$ is locally injective.

Since both $\Pcal \Zoll$ and $G^+_2(\R\Proj^3)$ are compact 5-dimensional manifolds and $T$ is locally injective, we conclude that $T$ is a covering map.
On the other hand, we know that $G^+_2(\R\Proj^3)\approx \R\Proj^3\times \Sph^2$, and that $\Pcal \Zoll$ is double covered by the connected space $\widetilde{\Pcal\Zoll}$ of the triples $(\widetilde{\Sigma},y,\mathcal{O})$ formed by $\widetilde{\Sigma} = \pi^{-1}(\Sigma)$ for some $\Sigma\in \Zoll$, $y\in \widetilde{\Sigma}$ and $\mathcal{O}$ an orientation for $T_y\widetilde{\Sigma}$.
Since $T$ lifts to a locally injective map $\widetilde{T}:(\widetilde{\Sigma},y,\mathcal{O})\in\widetilde{\Pcal\Zoll}\mapsto (y,T_y\widetilde{\Sigma})\in G^+_2(\Sph^3)$, and $G^+_2(\Sph^3)\approx \Sph^3\times \Sph^2$ is simply connected, we conclude that both $\widetilde{T}$ and $T$ are diffeomorphisms.

Since $T:\Pcal \Zoll\to G_2^+(\R\Proj^3)$ is a diffeomorphism, and the canonical projection $G_2^+(\R\Proj^3)\to G_2(\R\Proj^3)$ is a double cover, the composite $\Pcal \Zoll\to G_2(\R\Proj^3)$ is also a double cover.
Now, for each plane $\Sigma\in M$ and each point $x\in \Sigma$, there are two triples $(\Sigma,x,\mathcal{O})$ and $(\Sigma,x,-\mathcal{O})$ in $\Pcal M$, corresponding to the choice of orientation at $T_x\Sigma$.
Hence, for each pair $(x,V)\in G_2(\R\Proj^3)$ there is a unique $\Sigma\in \Zoll$ for which $T_x\Sigma = V$.
This proves that $\Zoll$ is a surface Zoll family in $(\R\Proj^3,g)$.


\section{Rigidity of the first four widths on surfaces}

This section is devoted to the proof of Theorem D, which states that a closed Riemannian surface $(M,g)$ is isometric to the real projective plane with its canonical metric if and only if it has equal first four widths.

We denote by $can$ the canonical metric on $\R\Proj^2$, and observe that, in \cite{Marx-Kuo}, the volume spectrum of $(\R\Proj^2,can)$ was completely calculated and the first five widths are equal to $2\pi$.
Hence, the only remaining implication to prove is that a closed surface is isometric to $(\R\Proj^2,can)$ if it has equal first four widths.

Let $(M,g)$ be a Riemannin surface with the first four widths equal to $2\pi$.
By Theorem 4.1 in \cite{AMN2}, $M$ is either $\Sph^2$ or $\R\Proj^2$.
In the first case, we also know that $g$ is a Zoll metric.
In both situations, there is no foliation $\{\gamma_t\}_{t\in S^1}$ of $(M,g)$ by simple closed geodesics.

Consider an optimal sequence of 4-sweepouts $\{\Psi_i\}\subset \Pcal_4$, in the sense that
\[ \lim_{i\to\infty} \max_{x\in X_i} \Mass \left( \Psi_i(x) \right) = 2\pi, \]
where $X_i = \text{dmn}(\Phi_i)$.
Let $\mathcal{V}$ be the set of stationary geodesic networks with integer multiplicities and mass equal to $2\pi$.
For $V\in \mathcal{V}$, let $T_V\in \I_1(M,\Z_2)$ be the sum of the edges of $V$ which have an odd multiplicity.
For $S\in\ZZ_1(M,\Z_2)$, we define
\[ d(S,\mathcal{V}) = \inf_{V\in\mathcal{V}}\left\{\Flat(S,T_V) + \FF(|S|,V) \right\}. \]
This satisfies the following triangle inequality:
\[ d(T,\mathcal{V})\leq d(S,\mathcal{V}) + \FF(T,S) \]
for every $T\in\ZZ_1(M,\Z_2)$.

Given $\delta>0$, we can suppose, by refining $X_i$ if necessary, that
\[ \FF(\Psi_i(x),\Psi_i(y)) < \delta/5 \]
for any $n$-dimensional simplex $\sigma\in X_i$ and any two points $x,y\in\sigma$.
Let $Z_i$ be the union of the simplices $\sigma\in X_i$ such that $d(\Psi_i(x),\mathcal{V})\leq 2\delta/5$ for every $x\in \sigma$, and $Y_i$ the union of all simplices $\tau\in X_i$ for which $d(\Psi_i(y),\mathcal{V})\geq \delta/5$ for every $y\in \tau$.
Then both $Y_i$ and $Z_i$ are sub-complexes of $X_i$ and $X_i = Y_i\cup Z_i$.

We claim that the restriction of $\Psi_i$ to $Y_i$ is not a sweepout for all $i$ sufficiently large.
To see this, suppose, for the sake of contradiction, that there is a subsequence $\{j\}\subset\{i\}$ for which $(\Psi_j|_{Y_j})^*(\overline\lambda)\neq 0$ in $H^1(Y_j;\Z_2)$.
Since $\omega_1(M,g) = 2\pi$ and since $\limsup_j\sup_{y\in Y_j} \Mass(\Psi_j(y)) \leq 2\pi$, we see that $\Psi_j|_{Y_j}$ is an optimal sequence of one-sweepouts for $\omega_1=2\pi$.
In particular, there is a one dimensional sub-complex $\sigma_j\subset Y_j$ that is homeomorphic to $S^1$ and such that $\Psi_j|_{\sigma_j}$ is a one-sweepout.

Define $\Psi'_j:=\Psi_j|_{\sigma_j}\in \Pcal_1$, and let $V\in C(\{\Psi'_j\})$.
If $V$ is not $\Z_2$-almost minimizing in annuli, we can proceed as in Parts 1 and 2 of the proof of Proposition 4.10 in \cite{Pitts} to obtain a set of concentric annuli so that, if a cycle $T$ is sufficiently close to $V$, say $\FF(V,|T|)<\epsilon = \epsilon(V)$, then $T$ must admit mass decreasing deformations supported in each annuli.

If, on the other hand, $V$ is $\Z_2$-almost minimizing in annuli, then it follows from $V$ being stationary and from Theorem 3.13 in \cite{Pitts} that $V$ is a stationary integral varifold.
Hence, $V\in\mathcal{V}$.
In this case, we denote by $\Tcal_V$ the set of cycles $T\in\ZZ_1(M,\Z_2)$ with support contained in the support of $V$.
By the Constancy Theorem for flat chains, every $T\in\Tcal_V$ is either $0$ or a sum of edges of $\text{spt}(V)$.

Therefore, $\Tcal_V$ is a finite set.
Moreover, if $T\in\Tcal_V$ is not equal to $T_V$, then  there is an edge $\gamma = \gamma(T,V)$ of $V$ with integer multiplicity $k\in\N$ where $T\llcorner\gamma = r\cdot \gamma$ for some $r\in\{0,1\}$ such that $r\neq k\mod 2$.
Now, the proof of Proposition 4.10 of \cite{MN-index} is local, and since $\gamma$ is smooth, it can be applied with $\Sigma = T$ to points $p\in\gamma$ near the midpoint of the edge $\gamma$.

Given $\rho>0$, there is some $\eta\in(0,\rho)$ such that, if $S\in \ZZ_1(M,\Z_2)$ and $\FF(|S|,V)<\eta$, then there is some $T\in\Tcal_V$ for which $\Flat(S,T)<\rho$.
If $2\rho<\delta/5$ and $S = \Psi'_j(x_j)$, $x_j\in\sigma_j$, then $T\neq T_V$ by the definition of $Y_j$.
Choose $\rho = \rho(V)>0$ sufficiently small so that Proposition 4.10 of \cite{MN-index} can e applied to $V$ and $\Sigma = T$ near the midpoint of the edge $\gamma(T,V)$ for any $T\in\Tcal_V$ not equal to $T_V$.
For each edge of V, this gives a set of annuli centered at its midpoint such that, if $S = |\Psi'_j(x_j)|$ ($x_j\in\sigma_j$) satisfies the inequality $\FF(|S|,V)<\epsilon(V)$, then one of the edges of $S$ admits mass decreasing deformations supported in each annulus.

By the compactness of $C(\{\Psi'_j\}_j)$ in the varifold topology, we can find a finite set $\{V_1,...,V_n\}\subset C(\{\Psi'_j\}_j)$ such that $C(\{\Psi'_j\}_j)\subset \cup_{k=1}^n B_\FF(V_k,\eta(V_k)/2)$.
Hence, for some positive number $\xi>0$ and for sufficiently large $j$, the varifold $|\Psi'_j(x)|$ is contained in some ball $B_\FF(V_k,\eta(V_k)/2)$ for any point $x\in\sigma_j$ such that $\Mass(\Psi'_j(x))\geq 2\pi-\xi$.

Thus, we must have the vanishing of the cohomology class $(\Psi_i|_{Y_i})^*(\overline\lambda) = 0$ for all sufficiently large $i$.
As a consequence, $(\Psi_i|_{Z_i})^*(\overline\lambda^3)\neq 0$ for all sufficiently large $i$.

Since $\delta>0$ was arbitrary, we can perform a diagonal argument to obtain a sequence of 3-sweepouts $\{\Phi_i\}_i$ such that
\begin{equation}
    \lim_{i\to\infty} \sup_{x\in X_i} d(\Phi_i(x),\mathcal{V}) = 0,
    \label{eq: stronger convergence for almgren-pitts}
\end{equation}
where $X_i = \text{dmn}(\Phi_i)$.
We can further assume that each $X_i$ is a connected 3-dimensional cubical complex.

Since $\{\Phi_i\}$ is a sequence of 3-sweepouts satisfying \eqref{eq: stronger convergence for almgren-pitts}, it is also a sequence of 2-sweepouts that satisfies the same equality.
And since we already know from \cite{AMN2} that $M$ is either $\Sph^2$ or $\R\Proj^2$, we can transpose all the arguments of \cite[Section 4]{AMN2} for our case.
In particular, we have the following result.

\begin{proposition}
    Given $V\in C(\{\Phi_i\})$, one of the three possibilities hold:
    \begin{itemize}
        \item[(i)] $V$ is a simple closed geodesic of length $2\pi$ with multiplicity one;
        \item[(ii)] $V$ is a sum of two simple closed geodesics of length $\pi$ and multiplicity one which intersect transversely at exactly one point; or
        \item[(iii)] $V$ is a simple closed geodesic of length $\pi$ and multiplicity two.
    \end{itemize}
    If \textnormal{(i)} holds, $M$ is diffeomorphic to $\Sph^2$, g is a Zoll metric, and every $W\in C(\{\Phi_i\})$ is also described as in \textnormal{(i)}.
    Otherwise, $M$ is $\R\Proj^2$ and every element $W\in C(\{\Phi_i\})$ is either as in \textnormal{(ii)} or as in \textnormal{(iii)}.
    \label{prop: Props 4.6 to 4.12 in AMN2}
\end{proposition}
\begin{proof}
    This is the summary of the content of \cite[Propositions 4.6 to 4.12]{AMN2}, and the argument that they do before \cite[Remark 4.1]{AMN2} and their proof of \cite[Theorem A]{AMN2} (see in particular the final case analysis in the proof of \cite[Theorem 4.1]{AMN2}).
    We can apply these arguments here, because $\{\Phi_i\}_i$ is also a sequence of 2-sweepouts. 
\end{proof}

Up to this point, we argued the same way as in the proof of \cite[Theorem A]{AMN2}.
Here is where our argument differs.

\begin{proposition}
    If $(M^2,g)$ is a Zoll sphere, then $\omega_4(M,g)>\omega_3(M,g)$.
    \label{prop: w_4>w_3 for Zoll spheres}
\end{proposition}
\begin{proof}
    Suppose, on the contrary, that $(M,g)$ is a Zoll sphere with $\omega_4(M,g)=\omega_3(M,g)$.
    After rescaling, we may assume that $\omega_1(M,g) = 2\pi$.
    By \cite[Proposition 3.2]{AMN2}, $\omega_4(M,g) = \omega_1(M,g)$, and we define $\mathcal{V}_1\subset\mathcal{V}$ as the set of stationary varifolds that are simple closed geodesics of length $2\pi$ and multiplicity one.
    By Proposition \ref{prop: Props 4.6 to 4.12 in AMN2}, and by the argument before it, there exists a sequence of 3-sweepouts $\{\Phi_i\}\subset \Pcal_3$ such that
    \[ \lim_{i\to\infty} \sup_{x\in X_i} d(\Phi_i(x),\mathcal{V}_1) = 0. \]
    In particular, $\Lambda(\{\Phi_i\}_i)=C(\{\Phi_i\}_i)\subset \mathcal{V}_1$.
    
    Observe that, since $T_V$ is also a simple closed geodesic of length $2\pi$ and multiplicity one for every $V\in\mathcal{V}_1$, we can think of $\mathcal{V}_1$ as both a set of varifolds and a set of mod-2 flat cycles.
    With this notation, $d(\Phi_i(x),\mathcal{V}_1)$ is identified with the $\FF$ metric in $\ZZ_1(M,\Z_2)$.
    In other words,
    \[ d(\Phi_i(x),\mathcal{V}_1) = \FF(\Phi_i(x),\mathcal{V}_1) \]
    for every $x\in X_i$ and we can rewrite the previous limit as
    \[ \lim_{i\to\infty} \sup_{x\in X_i} \FF(\Phi_i(x),\mathcal{V}_1) = 0. \]

    By Allard Regurality, we know that two elements of $\mathcal{V}_1$ that are sufficiently close in the varifold topology are close in the $C^\infty$ topology.
    Let $\Scal$ denote the set of smooth simple closed curves in $M$ with the $C^\infty$ topology.
    We may also write $\mathcal{V}_1\subset \Scal$ and denote an element of $\mathcal{V}_1$ by $\gamma$.
    
    We will prove that, for $i$ sufficiently large, we can construct 3-sweepouts $\Phi'_i:X_i\to \Scal\subset \ZZ_1(M,\Z_2)$ such that
    \[ \lim_{i\to\infty}\sup_{x\in X_i} \FF(\Phi_i'(x),\Phi_i(x)) = 0. \]
    This will give us the desired contradiction, since we know that $\Scal$ is homotopic to $\R\Proj^2$ by the Smale Theorem \cite{Smale}, and therefore such a $\{\Phi'_i\}$ could not be a sequence of 3-sweepouts.

    Let $\epsilon>0$ be chosen to ensure that two elements of $\mathcal{V}_1$ $\epsilon$-close to each other in the varifold topology are also graphically close, in the sense that one is the graph over the other.
    For given positive numbers $\epsilon_1$, $\epsilon_2$ and $\epsilon_3$, there are $\delta_1$, $\delta_2$ and $\delta_3$ such that, if $\gamma_0,\gamma_1\in \mathcal{V}_1$ satisfy $\FF(\gamma_0,\gamma_1)<\delta_i$, then $\gamma_1 = \text{graph}_{\gamma_0}(u)$ for some $u\in C^\infty(\gamma_0)$ and
    \[ \FF(\gamma_s,\gamma_t)<\epsilon_i \]
    for $\gamma_t:=\text{graph}_{\gamma_0}(tu)$ and $s,t\in[0,1]$.
    We take $0<\delta_1<\delta_2<\delta_3$ and $0<\epsilon_1<\epsilon_2<\epsilon_3<\epsilon$, and require that $4\delta_i/3+\epsilon_i< \delta_{i+1}$ for $i=1,2$ and $4\delta_3/3+\epsilon_3<\epsilon$.
    
    Fix $i_0>0$ such that $\FF(\Phi_i(x),\mathcal{V}_1)<\delta_1/3$ for all $x\in X_i$ and all $i\geq i_0$, and recall that $X_i = \text{dmn}(\Phi_i)$ is a 3-dimensional cubical complex.
    Subdivide $X_i$ as $X_i(l)$ so that $\FF(\Phi_i(x),\Phi_i(y))<\delta_1/3$ whenever $x$ and $y$ are contained in the same $k$-dimensional cell  $\sigma\in X_i(l)$ for any $k\in\{0,1,2,3\}$.
    
    For every vertex $x\in X_i(l)$ there is some $\gamma_x\in \mathcal{V}_1$ for which $\FF(\Phi_i(x),\gamma_x)<\delta_1/3$.
    If $[x,y]$ is an edge, then
    \[
        \begin{split}
            \FF(\gamma_x,\gamma_y) & \leq \FF(\gamma_x,\Phi_i(x)) + \FF(\Phi_i(x),\Phi_i(y)) + \FF(\Phi_i(y),\gamma_y) \\
            & < \delta_1,
        \end{split}
    \]
    so that $\gamma_y = \text{graph}_{\gamma_x}(u_y)$ for some smooth function $u_y$.
    We define $\Phi_i'$ on $[x,y]$ as
    \[ \Phi_i'((1-t)x+ty) = \text{graph}_{\gamma_x}(tu_y), \]
    and observe that
    \[ \FF(\Phi_i'(z),\Phi_i'(w))< \epsilon_1 \]
    whenever $z,w\in[x,y]$.
    This defines $\Phi_i'$ in the 1-skeleton on $X_i(l)$, and ensures the bound
    \[
        \begin{split}
            \FF(\Phi_i(z),\Phi_i'(z)) & \leq \FF(\Phi_i(z),\Phi_i(x)) + \FF(\Phi_i(x),\gamma_x) + \FF(\gamma_x,\Phi_i'(z)) \\
            & < 2\delta_1/3 + \epsilon_1
        \end{split}
    \]
    for all $z$ in an edge with one of the vertices $x$.

    To define $\Phi_i'$ over the 2-cells, we take, for any face $\sigma\in X_i(l)$, the center point $x_\sigma\in \sigma$, and choose $\gamma_\sigma\in \mathcal{V}_1$ such that
    \[ \FF(\Phi_i(x_\sigma),\gamma_\sigma)<\delta_1/3. \]
    If $y\in\partial \sigma$, then there is a bound
    \[
        \begin{split}
            \FF(\gamma_\sigma, \Phi_i'(y)) & \leq \FF(\gamma_\sigma,\Phi_i(x_\sigma)) + \FF(\Phi_i(x_\sigma),\Phi_i(y)) + \FF(\Phi_i(y),\Phi_i'(y)) \\
            & < 4\delta_1/3 + \epsilon_1 \\
            & < \delta_2.
        \end{split}
    \]
    Hence, $\Phi_i'(y) = \text{graph}_{\gamma_\sigma}(u_y)$ for some $u_y\in C^\infty(\gamma_\sigma)$ that varies continuously on $y\in\partial \sigma$.
    After identifying $\sigma$ with the disk $D^2\subset\R^2$ and $\partial\sigma\approx S^1=\partial D^2$, we define
    \[ \Phi_i'(ty) = \text{graph}_{\gamma_\sigma}(tu_y) \]
    for all $t\in[0,1]$ and all $y\in\partial\sigma$.
    This gives a continuous extension of $\Phi_i'$ to the 2-skeleton of $X_i(l)$, and we also have the upper bounds
    \[ \FF(\Phi_i(x),\Phi_i'(y)) < \epsilon_2 \]
    for all $x,y\in \sigma$ and all faces $\sigma\in X_i(l)$, and
    \[
        \begin{split}
            \FF(\Phi_i(y),\Phi_i'(y)) & \leq \FF(\Phi_i(y),\Phi_i(x_\sigma)) + \FF(\Phi_i(x_\sigma),\gamma_\sigma) + \FF(\gamma_\sigma,\Phi_i'(y)) \\
            & < 2\delta_1/3 + \epsilon_2 \\
            & <2\delta_2/3 + \epsilon_2.
        \end{split}
    \]

    Finally, we extend $\Phi_i'$ continuously to all 3-cells in the following way.
    For a fixed 3-cell $\sigma\in X_i(l)$, we take its center point $x_\sigma\in\sigma$ and choose a $\gamma_\sigma\in\mathcal{V}_1$ such that
    \[ \FF(\Phi_i(x_\sigma),\gamma_\sigma)<\delta_1/3 < \delta_2/3. \]
    If $y\in\partial \sigma$, then
    \[
        \begin{split}
            \FF(\gamma_\sigma, \Phi_i'(y)) & \leq \FF(\gamma_\sigma,\Phi_i(x_\sigma)) + \FF(\Phi_i(x_\sigma),\Phi_i(y)) + \FF(\Phi_i(y),\Phi_i'(y)) \\
            & < 4\delta_2/3 + \epsilon_2 \\
            & < \delta_3.
        \end{split}
    \]
    Hence, $\Phi_i'(y) = \text{graph}_{\gamma_\sigma}(u_y)$ for some $u_y\in C^\infty(\gamma_\sigma)$ that varies continuously on $y\in\partial \sigma$.
    After identifying $\sigma$ with the ball $B^3\subset\R^3$ and $\partial\sigma\approx \Sph^2=\partial B^3$, we define
    \[ \Phi_i'(ty) = \text{graph}_{\Sigma_\sigma}(tu_y) \]
    for all $t\in[0,1]$ and all $y\in\partial\sigma$.
    This gives a continuous extension of $\Phi_i'$ to all points of $X_i(l)$, and we have the upper bounds
    \[ \FF(\Phi_i(x),\Phi_i'(y)) < \epsilon_3 \]
    for all $x,y\in \sigma$ and all 3-cells $\sigma\in X_i(l)$, and
    \[
        \begin{split}
            \FF(\Phi_i(y),\Phi_i'(y)) & \leq \FF(\Phi_i(y),\Phi_i(x_\sigma)) + \FF(\Phi_i(x_\sigma),\gamma_\sigma) + \FF(\gamma_\sigma,\Phi_i'(y)) \\
            & < 2\delta_2/3 + \epsilon_3 \\
            & <2\delta_3/3 + \epsilon_3 \\
            & < \epsilon
        \end{split}
    \]
    for all $y\in\sigma$ and all 3-cells $\sigma\in X_i(l)$.
    Thus,
    \[ \FF(\Phi_i(x),\Phi_i'(x))<\epsilon \]
    for all $x\in X_i(l)$, and the existence of the sequence $\{\Phi_i\}_i$ is proved by taking $\epsilon>0$ arbitrarily small.

    Now, it is known that there is a positive number $\xi$ such that, if $\Phi:X\to\ZZ_1(M,\FF,\Z_2)$ is a 3-sweepout and $\Phi':X\to\ZZ_1(M,\FF,\Z_2)$ is a continuous map such that $\Flat(\Phi(x),\Phi'(x))<\xi$ for every $x\in X$, then $\Phi'$ is also a 3-sweepout.
    Thus, the map $\Phi'_i$ is a 3-sweepout for $i$ sufficiently large, and we arrive at the desired contradiction.
\end{proof}

\begin{proposition}
    Let $(M,g)$ be a closed Riemannian surface diffeomorphic to $\R\Proj^2$.
    If 
    \[ \omega_1(M,g) = \omega_2(M,g) = \omega_3(M,g) = \omega_4(M,g) = 2\pi, \]
    then $(M,g)$ is a Zoll surface whose geodesics have length $\pi$.
    \label{prop: w_1=w_4 for (RP^2,g) => Zoll}
\end{proposition}
\begin{proof}
    Let $\Gamma\subset \mathcal{V}$ be the set of varifolds that are either a sum of two simple closed geodesics of length $\pi$ and multiplicity one which intersect transversely at exactly one point, or a simple closed geodesic of length $\pi$ and multiplicity two.
    From Proposition \ref{prop: Props 4.6 to 4.12 in AMN2} and from the discussion preceding it, there is an optimal sequence of 3-sweepouts $\{\Phi_i\}$ such that
    \[ \lim_{i\to\infty} \sup_{x\in X_i} d(\Phi_i(x),\Gamma) = 0. \]
    In particular, $C(\{\Phi_i\}) = \Lambda(\{\Phi_i\})$ and $C(\{\Phi_i\})$ is a compact subset of $\Gamma$.
    We write $\Gamma_0:=C(\{\Phi_i\})\subset \Gamma$.

    Now, for each triple $\Omega_1$, $\Omega_2$ and $\Omega_3$ of open subsets of $M$ whose closures are mutually disjoint ($\overline{\Omega}_i\cap \overline{\Omega}_j = \emptyset$ for all $i\neq j$), and for each $i\in\N$, there is some $x_i\in X_i$ such that $\Phi_i(x_i)$ divides every $\Omega_i$ in two parts with the same area.
    Each $\Phi_i(x_i) =\partial U_i$ for $U_i\in \I_2(M,\Z_2)$, and we can take a subsequence, if necessary, and suppose that both $U_i\to U$ and $|\Phi_i(x_i)|\to \gamma_1+\gamma_2\in\Gamma_0$ as $i\to \infty$.
    Here, we denote by $\gamma_i$ a simple closed geodesic of length $\pi$.
    If $\gamma_1\neq\gamma_2$, then $\partial U = \gamma_1+\gamma_2$; otherwise, $U= M$ or $U = \emptyset$.
    Since $area(U\cap \Omega_j) = \lim_{i\to\infty}area(U_i\cap \Omega_j) = area(\Omega_j)/2$ for all $j=1,2,3$, we see that $U$ cannot be either $M$ or $\emptyset$, so that $\gamma_1\neq \gamma_2$.
    By shrinking $\Omega_j$ to points $p_j$, we conclude that for any triple of points $p_1,p_2,p_3\in M$ there is some $\gamma_1+\gamma_2\in \Gamma_0$ containing them (possibly with $\gamma_1 = \gamma_2$).

    We claim that through any two points $p,q\in M$ passes some simple closed geodesic $\gamma$ of length $\pi$.
    To see this, write $q = \exp_p(v)$ for some $v\in T_pM$, and consider a sequence $\gamma_1^n + \gamma_2^n\in \Gamma_0$ such that $p,q,\exp_p(v/n)\in \gamma_1^n+\gamma_2^n$.
    If $\gamma_1^n = \gamma_2^n$ for some $n$ or if either $p,q\in \gamma_1^n$ or $p,q\in \gamma_2^n$ for some $n$, we have nothing to prove.
    Hence, we can assume that those are not the cases for every $n$.
    Then either $p,\exp_p(v/n)\in \gamma_i^n$ or $\exp_p(v/n),q\in\gamma_i^n$ for some $i$ and infinitely many $n$'s.
    In the first case, we may take $i=1$ and pass to a subsequence to conclude that $\gamma^n_1\to\gamma$ as $n\to\infty$, where $\gamma$ is a simple closed geodesic of length $\pi$ that passes through $p$ with velocity $\lambda v$ for some $\lambda\neq 0$.
    Since $q = \exp_p(v)$, we conclude that $\gamma$ also passes through $q$.
    In the second case, we may also assume that $i=1$ and pass to a subsequence to conclude that $\gamma_1^n\to \gamma$ for some simple closed geodesic $\gamma$ of length $\pi$ and that passes through both $p$ and $q$. 

    To finish the proof, note that, through any point $p\in M$ and any vector $v\in T_pM$, we can take a sequence of simple closed geodesics of length $\pi$ $\gamma_n$ such that $p,\exp_p(v/n)\in\gamma_n$ for all $n$.
    Passing to a subsequence, if necessary, $\gamma_n$ converges to a simple closed geodesic $\gamma$ of length $\pi$ that can be parametrized in order to pass through $p$ with velocity $v$.
    Since both $p$ and $v$ were arbitrary, this shows that all geodesics of $(M,g)$ are simple closed curves of length $\pi$.
\end{proof}

\begin{proof}[Proof of Theorem D]
    Let $(M,g)$ be a closed surface with first four widths equal to $2\pi$.
    By \cite[Theorem 4.1]{AMN2} and Proposition \ref{prop: w_1=w_4 for (RP^2,g) => Zoll}, $(M,g)$ is either a Zoll sphere with geodesics of length $2\pi$ or a Zoll  real projective plane with geodesics of length $\pi$.
    Since $\omega_4(M,g) = \omega_3(M,g)$, the first case cannot happen by Proposition \ref{prop: w_4>w_3 for Zoll spheres}.
    Therefore, $(M,g)$ is a Zoll real projective plane whose geodesics have length $\pi$, and Green's Theorem \cite{Gre} implies that $(M,g)$ is isometric to $(\R\Proj^2,can)$.
\end{proof}

\end{document}